\documentclass[12pt]{article}
\usepackage{geometry}
\usepackage{amsmath}
\usepackage{amsthm}
\usepackage{bm}
\usepackage{graphicx}
\usepackage{cite}

\usepackage{fourier}
\usepackage{microtype}
\theoremstyle{plain}
\newtheorem{lemma}{Lemma}[section]

\newtheorem{theo}{Theorem}[section]

\newtheorem{assump}{\bf{Assumption}}[section]

\theoremstyle{remark}
\newtheorem{remark}{\bf{Remark}}[section]
\theoremstyle{remark}

\theoremstyle{remark}

\numberwithin{equation}{section}

\newcommand*\diff{\mathop{}\!\mathrm{d}}

\newcommand{\ie}{{\it i.e. \/}}

\newcommand{\R}{{\mathbb{R}}}
\newcommand{\Z}{{\mathbb{Z}}}
\newcommand{\N}{{\mathbb{N}}}

\newcommand{\p}{\mathcal{P}}

\newcommand{\dt}{\partial_t}

\newcommand{\dnz}[1]{\partial^{#1}_z}

\newcommand{\Dx}{\Delta x}
\newcommand{\Dt}{\Delta t}
\newcommand{\Dv}{\Delta v}

\newcommand{\tn}{t^n}

\newcommand{\x}{x_i}

\newcommand{\Uniz}{U^n_i(z)}
\newcommand{\Unpiz}{U^{n+1}_i(z)}
\newcommand{\Unimz}{U^n_{i-1}(z)}

\newcommand{\UniK}{U^n_{i,(K)}}

\newcommand{\bhUnpi}{\hat{\bm{U}}_i^{n+1}}

\newcommand{\bhUni}{\hat{\bm{U}}^n_i}

\newcommand{\bhUnim}{\hat{\bm{U}}^n_{i-1}}
\newcommand{\hUnik}{\hat{U}^n_{i,(k)}}

\newcommand{\lpm}{\lambda^{\pm}}
\newcommand{\lpmz}{\lambda^{\pm}(z)}

\begin{document}

\title{\bf The Discrete Stochastic Galerkin Method for Hyperbolic Equations
with Non-smooth and Random Coefficients\footnote{This research was supported by NSFC grant No. 91330203, NSF
grants DMS-1522184 and DMS-1107291: RNMS KI-Net, and by the Office
of the Vice Chancellor for Research and Graduate Education at the University of Wisconsin-Madison with
funding from the Wisconsin Alumni Research Foundation.}}

\vspace{1cm}
\author{Shi Jin\footnote{ Institute of Natural Sciences, Department of Mathematics, MOE-LSEC and SHL-MAC, Shanghai Jiao Tong University, Shanghai 200240, China and Department of Mathematics, University of Wisconsin-Madison, Madison, WI 53706, USA (jin@math.wisc.edu) and .} \ and Zheng Ma\footnote{Department of Mathematics, Shanghai Jiao Tong University, Shanghai 200240, China.}}

\maketitle

\begin{abstract}
\noindent 
We develop a general polynomial chaos (gPC) based stochastic Galerkin (SG) for
hyperbolic equations with random and singular coefficients. Due to the singular nature of the solution, the standard gPC-SG methods may suffer from a poor or even
non convergence. Taking advantage of the fact that the discrete solution,
by the central type  finite difference or finite volume approximations
in space and time for example, 
is smoother, we first discretize the equation by a smooth finite difference or
finite volume scheme, and then use the gPC-SG approximation to the discrete
system. The jump condition at the interface is treated using the immersed
upwind methods introduced in \cite{Jin:2009pro, Wen:2005ueba}.
This yields a method that converges with the spectral accuracy
for finite mesh size and time step.  We use a linear hyperbolic equation
with discontinuous and random coefficient, and the Liouville equation with
discontinuous and random potential, to illustrate our idea, with both one
and second order spatial discretizations. Spectral convergence is established 
for the first equation, and numerical examples for both equations show the
desired accuracy of the method.

\bigskip

\noindent{\bf Key words.} hyperbolic equation, random coefficient, potential barrier, stochastic Galerkin, polynomial chaos\\
{\bf AMS subject classifications.} 35L02, 65M06, 65M60, 65C30\\
\end{abstract}

\section{Introduction}\label{intro}
We are interested in developing efficient numerical methods to solve
linear hyperbolic equations with non-smooth and uncertain coefficients. Such problems arise in wave propagation in heterogeneous media, through interfaces between different media, or potential barriers,  making the coefficients in
these equations discontinuous or even more singular. Random uncertainties arise
due to modeling or experiment errors. These errors are inevitable since the
fluxes in  hyperbolic equations are often given by {\it empirical} laws, equations of state or
moment closures which are often {\it ad hoc}.

When hyperbolic equations contain singular coefficients, one usually needs to
provide an extra physical condition at the singular points to make the 
initial or boundary value problems well-posed and to account for 
the correct physics of waves at the interface or barrier~\cite{Wen:2005ueba, Jin:2009pro}.
In the case of potential barriers, a natural physical condition is the transmission
and reflection conditions,  and such conditions
can be built into the numerical fluxes in a natural way, in the framework of
the Hamiltonian-Preserving schemes~\cite{Wen:2005ueba, JinWen-wave}. 
This is the approach we will take to tackle the problems with singular coefficients.

To handle the difficulty induced by the random uncertainties, we will utilize the generalized
polynomial chaos (gPC) expansion based stochastic Galerkin (SG) method~\cite{Bijl:2013hkba, GS, GWZ, LMK, PIN, Tryoen:2010djba, Xiu:2010wxba, XiuKar}.
Such methods outperform the classical Monte-Carlo method in that, given
sufficient regularity of the solution in the random space, they can achieve
the spectral convergence, thus are  much more efficient  for
problems with random uncertainties. Unfortunately, for hyperbolic problems,
one often is not blessed with such regularities, which leads to 
significant reduction of order of
convergence~\cite{Motamed:2012gsba, Tang:2012ibba},
 thus slows down the computation or even gives 
non-convergent results due to Gibbs' phenomenon.  The problems under study in this paper are  problems with discontinuous solutions in the random space, due to jumps of solutions formed at the interfaces or barriers which will propagate
into the random space.

A standard gPC-SG method begins with a gPC approximation of the original
differential equation in the random space,  yielding a {\it deterministic} system of equations
for the gPC coefficients (while the randomness is built into the
basis functions which are orthonormal polynomials), which is then discretized
by standard schemes (finite difference, finite volume, finite element, or
spectral methods) in space and time. The gPC approximation
is accurate if solutions to the original problem are smooth in the random
space. This is not the case for the problems under study.

Our central idea in this paper is to {\it reverse} the above gPC-SG process.  Namely,
we first discretize the original equation in space and time, using {\it smooth} numerical fluxes, and then apply the gPC approximation to this discrete
equation.  Since the discrete solution is more regular than the continuous one,
the gPC approximation is applied to a smoother function (for fixed time step
and mesh  size), thus one expects a better convergence rate. We refer to
such gPC-SG methods as the {\it discrete gPC-SG methods}.

For hyperbolic equations, the smooth numerical fluxes are usually
central differences which do not depend on the characteristic information (for examples the Lax-Friedrichs, the Lax-Wendroff scheme, 
etc.). The upwind type schemes are not smooth, since they depend on the sign of
 the absolute value of the characteristic speeds thus do not yield smooth 
numerical fluxes.
 For second order scheme, in order to
suppress numerical viscosity, one usually uses slope limiters or ENO or
WENO type reconstruction~\cite{LeVeque, Tadmor:2000jlba, Tadmor:1990vlba} which in general are not smooth functions.
In order to keep the numerical flux smooth, we use the smooth BAP slope limiter
introduced in~\cite{Liu:1998faba}.

In this paper we will develop this idea for two problems. The first
is a scalar hyperbolic equation with a discontinuous and random coefficient:
\begin{equation}
\label{conv-eq}
  u_t(x, t, z) + \big[c(x,z)u(x,t,z)\big]_x = 0, \quad t > 0.
\end{equation}
Here $c(x,z)$ is the random coefficient where $z$ is a random variable in a properly defined complete random space with event space $\Omega$ and probability distribution function $\rho(z)$. $c(x,z)$ is discontinuous respect to $x$, which corresponds to an interface between different media. The second is the Liouville equation for the particle density distribution $u(x,t,z)>0$:
\begin{equation}\label{eq_Liou}
  u_t + v  u_x - V_x  u_v = 0, \quad t>0, \quad x, v \in \R,
\end{equation}
in which the potential function $V(x,z)$ may be discontinuous in $x$, corresponding to a potential barrier. 
The quantities of interest to be computed in these problems include the expectation of $u$,
\begin{equation}\label{exp}
  \mathbb{E}[u] = \int u(z) \rho(z)\diff z.
\end{equation}
and its variance
\begin{equation}\label{var}
  \mathbb{V}[u]:=\mathbb{E}\big[(u-\mathbb{E}(u))^2\big] = \int u(z)^2  \rho(z) \diff z - (\mathbb{E}[u])^2
\end{equation}

For equation~(\ref{conv-eq}), by using the Lax-Friedrichs scheme followed
by the gPC-SG approximation,  we will establish the regularity and consequently the spectral convergence of the proposed method in
the random space, while the numerical convergence in space and time is the
same as the deterministic problem established in~\cite{Qi:2013byba}.
The error will be verified numerically, for both the convection and the Liouville equations. 

In such problems the uncertainty may also come from the initial data. This is
a well-studied problem~\cite{GX, Tang:2012ibba} and our method can obviously be used in this case.

The paper is organized as follows. In Section 2, we will present the
discrete gPC-SG method for  the convection equation~(\ref{conv-eq}),
and conduct the regularity and numerical convergence analysis for the fully
discrete scheme.  In Section 3, we will show how to use this idea for the Liouville equation for both first  and  second order spatial discretizations.
 In Section 4, we will present numerical examples for both equations that will show an exponential convergence in the random space.

\section{A Discrete gPC-SG scheme for convection equation with discontinuous wave speed}


We first consider a scalar model convection equation 
\begin{equation}\label{eq:convection}
  \left\{
  \begin{aligned}
    &u_t(x,t,z) + \big[c(x,z)u(x,t,z)\big]_x = 0,\quad t>0, \\
    &u(x, 0, z) = u_0(x, z).
  \end{aligned}
  \right.
\end{equation}
Here we consider the case that $c(x, z)$ can be discontinuous with respect to $x$ at some point, for example,
\begin{equation}
  c(x, z) = 
  \begin{cases}
    c^-(z)>0, & \text{if $x<0$}, \\
    c^+(z)>0, & \text{if $x>0$}.
  \end{cases}
\end{equation}
As in \cite{Jin:2009pro}, an interface condition at $x = 0$ is needed to make the problem well-posed:
\begin{equation}\label{interface_1}
  u(0^-,t, z) = \alpha(z)u(0^+,t,z).
\end{equation}
where $\alpha(z) = 1$ corresponds to conservation of mass or $\alpha(z) = c^-(z) / c^+(z)$ for the conservation of flux which is the case we will use in the sequel. Notice that here we assume $c(x,z)$ is smooth enough with respect to the random variable $z$, and only has one discontinuous point at $x = 0$.

The discontinuity of $u(x,t,z)$ generated by the interface condition (\ref{interface_1}) will propagate into the random space, preventing the gPC method from
 high order convergence due to Gibb's phenomenon. Here we propose a slightly different approach from the traditional gPC method: We first discretize equation (\ref{eq:convection}) in space and time as done in~\cite{Qi:2013byba} with the random variable $z$ as a fixed parameter. A key idea in~\cite{Qi:2013byba} is to ``immerse'' the interface condition~(\ref{interface_1}) into the scheme. The gPC method will then be applied to the discrete system.

\subsection{The scheme}

Let the spatial mesh be $\x = i\Dx$, where $i\in\Z$, the set of all integers, and $\Dx$ is the mesh size. Let $\tn = n\Dt$ be the discrete time where $\Dt$ is the time step. Let $\Uniz = U(\x,\tn,z)$ be the numerical approximation of $u(\x, \tn, z)$. The immersed upwind scheme proposed by Jin and Qi in \cite{Qi:2013byba}, for (\ref{eq:convection}) (\ref{interface_1}) is
\begin{equation}\label{dis_scheme}
  \left\{
  \begin{aligned}
    \Unpiz &= (1 - \lambda^-(z))\Uniz+\lambda^-(z)\Unimz, &\quad&\text{if $i\leq 0$}, \\
    \Unpiz &= (1 - \lambda^+(z))\Uniz+\lambda^-(z)\Unimz, &&\text{if $i = 1$}, \\
    \Unpiz &= (1 - \lambda^+(z))\Uniz+\lambda^+(z)\Unimz, &&\text{if $i\geq 2$}, 
  \end{aligned}
  \right.
\end{equation}
where $\lambda^{\pm}(z) = c^{\pm}(z) \Dt / \Dx$.

Notice, from this discrete scheme (\ref{dis_scheme}), if one assumes that 
$\Uniz$ is a smooth function of $z$ for each $i$, then after one time step,  $\Unpiz$ is still a smooth function of $z$. The reason is simple: $\lpmz = c^{\pm}(z) \Dt / \Dx$ is a smooth function of $z$! Since we assume the initial data is smooth with respect to $z$, the numerical solution at any time $\tn$ should also be smooth with respect to $z$. Then if applying the standard gPC Galerkin method to this discrete system, one can expect a fast convergence of gPC expansion to this discrete solution when the physical mesh size $\Dx$ and $\Dt$ are fixed. 

 First we recall the scheme:
  \begin{equation}\label{dis_scheme_p}
    \left\{
    \begin{aligned}
      \Unpiz &= (1 - \lambda^-(z))\Uniz+\lambda^-(z)\Unimz &\quad& \text{if $i\leq 0$}, \\
      \Unpiz &= (1 - \lambda^+(z))\Uniz+\lambda^-(z)\Unimz && \text{if $i = 1$}, \\
      \Unpiz &= (1 - \lambda^+(z))\Uniz+\lambda^+(z)\Unimz && \text{if $i\geq 2$}. 
    \end{aligned}
    \right.
  \end{equation}
following the standard gPC Galerkin framework, we apply the gPC expansion of $z$ to $\Uniz$. Namely, we seek an approximate solution in the form of gPC expansion, \ie
\begin{equation}\label{gPC_exp}
  \UniK(z) = \sum_{k=0}^K \hUnik P_k(z),
\end{equation}
where $P_k(z)$ form an orthonormal polynomial basis with weights $\rho(z)$, and the degree of $P_k(z)$ of $k$ satisfying 
\begin{equation}
  \langle P_i, P_j \rangle = \int P_i(z)P_j(z)\rho(z)\diff z = \delta_{ij},
\end{equation}
with the weighted inner product defined as
\begin{equation}
  \langle f, g \rangle = \int f(z) g(z)\rho(z)\diff z,
\end{equation}
and $\delta_{ij}$ is the Kronecker delta function. The expansion coefficients are determined as
\begin{equation}
  \hUnik = \int U^n_i(z)P_k(z)\rho(z)\diff z.
\end{equation}

By utilizing the expansion~(\ref{gPC_exp}) and employing a Galerkin projection,  the coefficients $\hUnik$ satisfy the following system of equations
\begin{equation}\label{dis_gPC}
  \left\{
  \begin{aligned}
    \bhUnpi &= (1 - \bm{\lambda}^-) \bhUni + \bm{\lambda}^- \bhUnim, &\quad&\text{if $i\leq 0$}, \\
    \bhUnpi &= (1 - \bm{\lambda}^+) \bhUni + \bm{\lambda}^- \bhUnim, &&\text{if $i=1$}, \\
    \bhUnpi &= (1 - \bm{\lambda}^+) \bhUni + \bm{\lambda}^+ \bhUnim, &&\text{if $i\geq 2$}. 
  \end{aligned}
  \right.
\end{equation}
Here $\bhUni = (\hat{U}^n_{i, (0)}, \dotsc, \hat{U}^n_{i, (K)})^T$ is a
vector of dimension $(K+1)$, and $\bm{\lambda}^{\pm}$ are the $(K+1) \times (K+1)$ matrices whose entries are $\{\lambda^{\pm}_{k,m}\}_{0 \leq k, m \leq K}$ where
\begin{equation}
  \lpm_{k,m} = \int c^{\pm}(z) P_k(z) P_m(z) \rho(z)\diff z\,.
\end{equation}

\subsection{The  error estimate and convergence analysis}


We first introduce some notations, spaces and norms that will be used for
our analysis. We assume that $u(x,t,z)$ has a compact support in the domain $D = [a, b]$, where $a < 0$ and $b > 0$ such that the domain includes the interface $x = 0$. $-M\leq i\leq M$ is the spatial discretization index and $\Dx = (b - a)/(2M+1)$. The time step index is $n = 0, 1, \dotsc$. 

Define a weighed $L^2$ norm on the random space $\Omega$,
\begin{equation}
  \|f(\cdot)\|^2_{L^2(\Omega)} = \int f^2(z)\rho(z)\diff z.
\end{equation}
We also define the norm
\begin{equation}
  \|u^n(\cdot)\|_H^2 := \int\|u^n(z)\|_{l^1(D)}^2\rho(z)\diff z,
\end{equation}
where 
\begin{equation}
  \|u^n(z)\|_{\ell^1(D)} = \sum_{i=-M}^M|u^n_i(z)|\Dx.
\end{equation}

\subsubsection{Regularity of the discrete solution in the random space}

In order to obtain the error estimate, we need to investigate the regularity of discrete solution
$\Uniz$ in the random space. It is natural that some assumptions for the given data will be made. More precisely, we make the following assumptions (see~\cite{Tang:2012ibba,Zhou:2010gwba}).
\begin{assump}
  \begin{equation}\label{assu}
    \underset{z\in\Omega}{\max}|\dnz{s} \lambda^{\pm}(z)|\leq\gamma_\ell, \quad\underset{D\otimes\Omega}{\max}|\dnz{s} u_0(x,z)|\leq \eta_\ell, \quad \forall 0\leq s\leq\ell,
  \end{equation}
  where $0\leq\lambda^{\pm}(z)=c^{\pm}(z)\Dt/\Dx\leq 1$ and $\ell = 1, 2, \dotsc$ and $\gamma$, $\eta$ are positive constants. Without loss of generality, we also assume a bounded constant $\tau = \max\big\{\gamma_\ell, \eta_\ell, 1\big\}$.
\end{assump}

Note that in (\ref{assu}) the constants $\gamma_\ell$ and $\eta_\ell$ are independent of $x$. We are now ready to state and prove the following regularity result.
\begin{theo}\label{regularity}
  Under Assumption (\ref{assu}), the discrete solution $U^n_i(z)$ have properties
  \begin{equation}
\label{reg}
    \underset{i\in\N, z\in\Omega}{\max}|\dnz{\ell} \Uniz|\leq C_\ell(n)(2\tau)^n \tau,
  \end{equation}
  for $\forall\ell\in\N$, where
  \begin{equation}\label{cln}
    C_\ell(n) = \sum\limits^n_{s=0}\binom{n}{s}(1+s)^\ell\leq 2^{(\ell+1)n}.
  \end{equation}
\end{theo}
\begin{proof}
  Differentiating scheme (\ref{dis_scheme}) $\ell$ times with respect to $z$,
  \begin{equation*}
    \left\{
    \begin{aligned}
      \dnz{\ell}\Unpiz &= \dnz{\ell}\Uniz - \sum\limits_{s=0}^\ell\binom{\ell}{s}\dnz{\ell-s} \lambda^-(z)\dnz{s}\Uniz+\sum\limits_{s = 0}^\ell\binom{\ell}{s}\dnz{\ell-s} \lambda^-(z)\dnz{s}\Unimz &\quad&\text{if $i\leq 0$}, 
      \\
      \dnz{\ell}\Unpiz &= \dnz{\ell}\Uniz  - \sum\limits_{s = 0}^l\binom{\ell}{s}\dnz{\ell-s} \lambda^+(z)\dnz{s}\Uniz+\sum\limits_{s = 0}^l\binom{\ell}{s}\dnz{\ell-s} \lambda^-(z)\dnz{s}\Unimz &&\text{if $i=1$}, 
      \\
      \dnz{\ell}\Unpiz &= \dnz{\ell}\Uniz  - \sum\limits_{s = 0}^\ell\binom{\ell}{s}\dnz{\ell-s} \lambda^+(z)\dnz{s}\Uniz+\sum\limits_{s = 0}^l\binom{\ell}{s}\dnz{\ell-s} \lambda^+(z)\dnz{s}\Unimz && \text{if $i\geq 2$}.
    \end{aligned}
    \right.
  \end{equation*}
  We will use the mathematical induction on the index $n$. When $n = 1$ which means after the first step, one has
  \begin{equation*}
    \left\{
    \begin{aligned}
      \dnz{\ell} U^1_i (z) &= \dnz{\ell} U^0_i(z) - \sum\limits_{s=0}^\ell\binom{\ell}{s}\dnz{\ell-s} \lambda^-(z)\dnz{s} U^0_i(z)+\sum\limits_{s = 0}^\ell\binom{\ell}{s}\dnz{\ell-s} \lambda^-(z)\dnz{s} U^0_{i-1}(z) &\quad&\text{if $i\leq 0$}, 
      \\
      \dnz{\ell} U^1_i (z) &= \dnz{\ell} U^0_i(z)  - \sum\limits_{s = 0}^l\binom{\ell}{s}\dnz{\ell-s} \lambda^+(z)\dnz{s} U^0_i(z)\sum\limits_{s = 0}^l\binom{\ell}{s}\dnz{\ell-s} \lambda^-(z)\dnz{s} U^0_{i-1}(z) &&\text{if $i=1$}, 
      \\
      \dnz{\ell} U^1_i (z) &= \dnz{\ell} U^0_i(z)  - \sum\limits_{s = 0}^\ell\binom{\ell}{s}\dnz{\ell-s} \lambda^+(z)\dnz{s} U^0_i(z)+\sum\limits_{s = 0}^l\binom{\ell}{s}\dnz{\ell-s} \lambda^+(z)\dnz{s} U^0_{i-1}(z) && \text{if $i\geq 2$}.
    \end{aligned}
    \right.
  \end{equation*}
  With Assumption (\ref{assu}),
  \begin{equation}
    \underset{i\in\N, z\in\Omega}{\max}|\dnz{s} U^0_i(z)|= \underset{i\in\N, z\in\Omega}{\max}|\dnz{s} u_0(x_i,z)|\leq\underset{D\otimes\Omega}{\max}|\dnz{s} u_0(x,z)|\leq \tau,
  \end{equation}
  and
  \begin{equation}
    \underset{z\in\Omega}{\max}|\dnz{\ell-s} \lambda^{\pm}(z)|\leq\tau.
  \end{equation}
  So one has
  \begin{equation}
    \begin{split}
      \underset{i\in\N, z\in\Omega}{\max}|\dnz{\ell} U^1_i(z)|\leq & {} \underset{i\in\N, z\in\Omega}{\max}|\dnz{\ell} U^0_i(z)| 
      + \sum_{s = 0}^\ell\binom{\ell}{s}\underset{z\in\Omega}{\max}|\dnz{\ell-s} \lambda^{\pm}(z)|\underset{i\in\N, z\in\Omega}{\max}|\dnz{s} U^0_i(z)|
      \\
      & {} + \sum_{s = 0}^\ell\binom{\ell}{s}\underset{z\in\Omega}{\max}|\dnz{\ell-s} \lambda^{\pm}(z)|\underset{i\in\N, z\in\Omega}{\max}|\dnz{s} U^0_{i-1}(z)| 
      \\
      \leq & {} \tau + 2\tau^2\sum_{s=0}^l\binom{\ell}{s}
      \leq  {} 2\tau(2^\ell + 1)\tau,
    \end{split}
  \end{equation}
  which satisfies (\ref{reg}) for $n=1$.
  
  Next we assume when $n = p$, the derivatives satisfy (\ref{reg}):
  \begin{equation}
    \underset{i\in\N, z\in\Omega}{\max}|\dnz{\ell} U^p_i(z)|\leq C_\ell(p)(2\tau)^p\tau, \quad\forall \ell\in\N.
  \end{equation}  
  Then for index $n = p+1$, using the same procedure as above,
  \begin{equation}
    \begin{split}
      \underset{i\in\N, z\in\Omega}{\max}|\dnz{\ell} U^{p+1}_i(z)| \leq & {} \underset{i\in\N, z\in\Omega}{\max}|\dnz{\ell} U^p_i(z)| + \sum_{s = 0}^\ell\binom{\ell}{s}\underset{z\in\Omega}{\max}|\dnz{\ell-s} \lambda^{\pm}(z)|\underset{i\in\N, z\in\Omega}{\max}|\dnz{s} U^p_i(z)|
      \\
      & {} + \sum_{s = 0}^\ell\binom{\ell}{s}\underset{z\in\Omega}{\max}|\dnz{\ell-s} \lambda^{\pm}(z)|\underset{i\in\N, z\in\Omega}{\max}|\dnz{s} U^p_{i-1}(z)|
      \\
      \leq & {} C_\ell(p)(2\tau)^p\tau + 2\sum_{s=0}^\ell\binom{\ell}{s}\tau C_{\ell-s}(p)(2\tau)^p\tau
      \\
      \leq & {} \left(C_\ell(p) + \sum_{s=0}^\ell\binom{\ell}{s}C_{\ell-s}(p)\right)(2\tau)^{p+1}\tau
      \\
      := & {} C_\ell(p+1)(2\tau)^{p+1}\tau.
    \end{split}
  \end{equation}
  From the last equality one gets the recursive relation of $C_\ell(p)$,
  \begin{equation}
    C_\ell(n+1) = C_\ell(n) + \sum\limits^\ell_{s=0}\binom{\ell}{s}C_{\ell-s}(n),
  \end{equation}
  and by the mathematical induction one can find
  \begin{equation}
    C_\ell(n) = \sum\limits^n_{s=0}\binom{n}{s}(1+s)^\ell,
  \end{equation}
  which is the desired result.
\end{proof}
\begin{remark}
  The coefficient
  \begin{equation}\label{Cnl}
    C_\ell(n) = \sum\limits^n_{s=0}\binom{n}{s}(1+s)^\ell \leq 2^n(1+n)^\ell\leq 2^{(\ell+1)n}.
  \end{equation}
  For a given final time $T = n\Dt$,
  \begin{equation}
    C_\ell(n)\leq 2^{\frac{T}{\Dt}}\left(1+\frac{T}{\Dt}\right)^\ell\leq 2^{\frac{(\ell+1)T}{\Dt}}.
  \end{equation}
\end{remark}

\subsubsection{The  spectral convergence of the gPC Galerkin method}
Let $\Uniz$ be the solution to the linear convection equation~(\ref{dis_scheme}). We define the $K$th order projection operator
\begin{equation}
  \p_K \Uniz = \sum_{k=0}^K \big\langle \Uniz, P_k(z) \big\rangle P_k(z).
\end{equation}
The error arisen from the gPC-SG can be split into two parts $r^n_{i,(K)}(z)$ and $e^n_{i,(K)}(z)$,
\begin{equation}\label{split}
  \begin{split}
    \Uniz - \UniK(z) & = \Uniz - \p_K \Uniz + \p_K \Uniz - \UniK(z)\\
    & := r^n_{i,(K)}(z) + e^n_{i,(K)}(z),
  \end{split}
\end{equation}
where $r^n_{i,(K)}(z) = \Uniz - \p_K \Uniz$ is the interpolation error, and $e^n_{i,(K)}(z)=\p_K \Uniz - \UniK(z)$ is the projection error.

For the interpolation error $r^n_{i,(K)}(z)$, we have the following lemma,
\begin{lemma}[{\bf Interpolation error}]\label{tr_e}
  Under  Assumption (\ref{assu}),  for a given final time $T = n\Dt$ and any given integer $\ell\in\N$,
  \begin{equation}
    \|r^n_{\cdot,(K)}(\cdot)\|_H \leq \frac{(b-a)C_\rho (2^{\ell+2}\tau)^n\tau}{K^\ell}, \quad\forall\ell\in\N,
  \end{equation}
  where $C_\rho$ is a constant depends on the orthogonal polynomials $\{P_k(z)\}_{k\in\N}$ .
\end{lemma}
\begin{proof}
  By the definition of $r^n_{i,(K)}(z)$ and the norm $\|\cdot\|_H$,
  \begin{equation}
    \begin{split}
      \|r^n_{\cdot,(K)}(\cdot)\| 
      & = \|U^n_\cdot(\cdot) - \p_K U^n_\cdot(\cdot)\|_H 
      \\
      & = \Bigg(\int\|U^n_\cdot(z) - \p_K U^n_\cdot(z)\|^2_{l^1(D)}\rho(z)\diff z\Bigg)^{1/2}
      \\
      & = \Bigg(\int\Big(\sum\limits_{i=-M}^M |U^n_i(z) - \p_K U^n_i(z)|\Dx\Big)^2\rho(z)\diff z\Bigg)^{1/2}
      \\
      & \leq \sum\limits_{i=-M}^M \Bigg(\int |U^n_i(z) - \p_K U^n_i(z)|^2\rho(z)\diff z\Bigg)^{1/2}\Dx
      \\
      & = \sum_{i=-M}^M \|U^n_i(\cdot) - \p_K U^n_i(\cdot)\|_{L^2(\Omega)}\Dx,
    \end{split}
  \end{equation}
  here we have used the Minkowski inequality. Then by the standard error estimate for orthogonal polynomial approximations~\cite{Quarteroni:1982bmba}, we get
  \begin{equation}
    \|U^n_i(\cdot) - \p_K U^n_i(\cdot)\|_{L^2(\Omega)}\leq \frac{C_\rho\|\dnz{\ell}\Uniz\|_{L^2(\Omega)}}{K^\ell}.
  \end{equation}
  By using {\bf Theorem~\ref{regularity}}, one obtains
  \begin{equation}
    \|\dnz{\ell}U^n_i(\cdot)\|_{L^2(\Omega)} \leq \underset{i\in\N, z\in\Omega}{\max}\big|\dnz{\ell}\Uniz\big| \Bigg(\int \rho(z)\diff z\Bigg)^2 \leq C_l(n)(2\tau)^n\tau \leq 2^{(\ell+1)n}(2\tau)^n\tau,
  \end{equation}
  for $\forall l\in\N$, then 
  \begin{equation}\label{err2}
    \|U^n_i(\cdot) - \p_K U^n_i(\cdot)\|_{L^2(\Omega)}\leq C_\rho (2^{\ell+2}\tau)^n\tau/K^\ell, \quad\forall\ell\in\N,
  \end{equation}
  which leads to
  \begin{equation}
      \|U^n_\cdot(\cdot) - \p_K U^n_\cdot(\cdot)\|_H \leq \sum_{i=-M}^M C_\rho (2^{\ell+2}\tau)^n\tau/K^\ell\Dx = \frac{(b-a)C_\rho (2^{\ell+2}\tau)^n\tau}{K^\ell}.
  \end{equation}
  This completes the proof.
\end{proof}

It remains to estimate $e^n_{i,(K)}(z)$. To this aim,  first notice that $\UniK(z)$ satisfies
\begin{equation}\label{uk}
  \left\{
  \begin{aligned}
    U^{n+1}_{i,(K)}(z) &= \UniK(z) - \p_K \big[\lambda^-(z)\big(\UniK(z) - U^n_{i-1,(K)}(z)\big)\big] &\quad&\text{if $i\leq 0$}, \\
    U^{n+1}_{i,(K)}(z) &= \UniK(z) - \p_K \big[\big(\lambda^+(z)\UniK(z) - \lambda^-(z)U^n_{i-1,(K)}(z)\big)\big] && \text{if $i = 1$}, \\
    U^{n+1}_{i,(K)}(z) &= \UniK(z) - \p_K \big[\lambda^+(z)\big(\UniK(z) - U^n_{i-1,(K)}(z)\big)\big] && \text{if $i\geq 2$}.
  \end{aligned}
  \right.
\end{equation}
On the other hand, by doing the $K$th order projection directly on the scheme~(\ref{dis_scheme}), one obtains
\begin{equation}\label{pu}
  \left\{
  \begin{aligned}
    \p_K\Unpiz &= \p_K\Uniz - \p_K\big[\lambda^-(z)\big(\Uniz - \Unimz)\big] &\quad& \mbox{if } i\leq 0, \\
    \p_K\Unpiz &= \p_K\Uniz - \p_K\big[\big(\lambda^+(z)\Uniz - \lambda^-(z)\Unimz)\big] && \text{if $i = 1$}, \\
    \p_K\Unpiz &= \p_K\Uniz - \p_K\big[\lambda^+(z)\big(\Uniz - \Unimz)\big] && \text{if $i\geq 2$}. 
  \end{aligned}
  \right.
\end{equation}
(\ref{pu}) subtracted by~(\ref{uk}) gives
\begin{equation}\label{ek}
  \left\{
  \begin{aligned}
    e^{n+1}_{i,(K)} = e^n_{i,(K)} & - \p_K \big[\lambda^-(z)(e^n_{i,(K)} - e^n_{i-1,(K)})\big] \\
                    & - \p_K \big[\lambda^-(z)(r^n_{i,(K)} - r^n_{i-1,(K)})\big] &\quad&\text{if $i\leq 0$}, \\
    e^{n+1}_{i,(K)} = e^n_{i,(K)} &- \p_K \big[(\lambda^+(z)e^n_{i,(K)} - \lambda^-(z)e^n_{i-1,(K)})\big] \\
                    & - \p_K \big[(\lambda^+(z)r^n_{i,(K)} - \lambda^-(z)r^n_{i-1,(K)})\big]&& \text{if $i = 1$}, \\
    e^{n+1}_{i,(K)} = e^n_{i,(K)} &- \p_K \big[\lambda^+(z)(e^n_{i,(K)} - e^n_{i-1,(K)})\big] \\
                    & - \p_K \big[\lambda^+(z)(r^n_{i,(K)} - r^n_{i-1,(K)})\big] && \text{if $i\geq 2$}.
  \end{aligned}
  \right.
\end{equation}
where  the variable $z$ is omitted for clarity.

Now we can give the following estimate of the projection error $e^n_{i,(K)}(z)$,
\begin{lemma}[{\bf Projection error}]\label{pr_e}
  Under Assumption (\ref{assu}),  for a given final time $T = n\Dt$ and any given integer $\ell\in\N$ that the projection error satisfies the following estimate,
  \begin{equation}
    \|e^n_{\cdot,(K)}(\cdot)\|_H \leq \frac{2\tau(b-a)C_\rho C'_\ell(n)}{K^\ell}, \quad \forall\ell\in\N,
  \end{equation}
  where $C'_{\ell}(n)= \dfrac{(2^{\ell+2}\tau)^n - 3^n}{2^{\ell+2}\tau - 3}$ and $C_\rho$ is a constant determined only by the orthogonal polynomials $\{P_k(z)\}_{k\in\N}$ .
\end{lemma}
\begin{proof}
  First, according to~(\ref{ek}), one has the following estimate for $i\leq 0$,
  \begin{equation}
    \begin{split}
      \|e^{n+1}_{i,(K)}\|_{L^2(\Omega)} & \leq \|e^n_{i,(K)}\|_{L^2(\Omega)} + \|\p_K\| \big[\max\limits_{z\in\Omega}(\lambda^-(z))(\|e^n_{i,(K)}\|_{L^2(\Omega)} + \|e^n_{i-1,(K)}\|_{L^2(\Omega)})\big] \\
      &\quad + \|\p_K\| \big[\max\limits_{z\in\Omega}(\lambda^-(z))(\|r^n_{i,(K)}\|_{L^2(\Omega)} + \|r^n_{i-1,(K)}\|_{L^2(\Omega)})\big].
    \end{split}
  \end{equation}
Note   $\|\p_K\|\leq 1$ since it is a projection operator and $\max\limits_{z\in\Omega}(\lambda^{\pm}(z))\leq 1$, so one gets
  \begin{equation}
    \begin{split}
      \|e^{n+1}_{i,(K)}\|_{L^2(\Omega)} & \leq \|e^n_{i,(K)}\|_{L^2(\Omega)} + \|e^n_{i,(K)}\|_{L^2(\Omega)} + \|e^n_{i-1,(K)}\|_{L^2(\Omega)} \\
      &\quad + \|r^n_{i,(K)}\|_{L^2(\Omega)} + \|r^n_{i-1,(K)}\|_{L^2(\Omega)}.
    \end{split}
  \end{equation}
  According to (\ref{err2}), 
  \begin{equation}
    \|r^n_{i,(K)}\|_{L^2(\Omega)}\leq C_\rho (2^{\ell+2}\tau)^n\tau/K^\ell, \quad\forall i\in\Z, \forall \ell\in\N,
  \end{equation}
  so
  \begin{equation}
    \|e^{n+1}_{i,(K)}\|_{L^2(\Omega)} \leq 2\|e^n_{i,(K)}\|_{L^2(\Omega)} + \|e^n_{i-1,(K)}\|_{L^2(\Omega)} + C_\rho 2\tau(2^{\ell+2}\tau)^n/K^\ell.
  \end{equation}

Similarly, for $i = 1$ and $i\geq 2$, one has the same estimate as above.
Summing over $i$ and multiplying by $\Dx$ give
  \begin{equation}
    \|e^{n+1}_{(K)}\|_H \leq  3\|e^n_{(K)}\|_H + 2\tau(b-a)C_\rho (2^{\ell+2}\tau)^n /K^\ell.
  \end{equation}
  Using this recursive relation and notice that $\|e^0_{(K)}\|_H=0$, one obtains
  \begin{equation}
    \|e^n_{(K)}\|_H \leq \dfrac{2\tau(b-a)C_\rho}{K^\ell}\dfrac{(2^{\ell+2}\tau)^n - 3^n}{2^{\ell+2}\tau - 3}:=\frac{2\tau(b-a)C_\rho C'_\ell(n)}{K^\ell}.
  \end{equation}
  This completes the proof of the lemma.
\end{proof}

We are now ready to state the convergence theorem of gPC-SG method for the discrete scheme:
\begin{theo}\label{gpc_conv}
  Under  Assumption (\ref{assu}),  for a given final time $T = n\Dt$ and any given integer $\ell\in\N$, the error of the gPC-SG method for the discrete scheme is
  \begin{equation}\label{err}
    \|U^n - U^n_{(K)}\|_H \leq \frac{(b-a)C_\rho C(\ell,n)}{K^\ell}, \quad\forall\ell\in\N,
  \end{equation}
  where $C(\ell, n) = (2^{\ell+2}\tau)^n\tau + 2 \tau C'_\ell(n)$.
\end{theo}
\begin{proof}
  From {\bf Lemma~\ref{tr_e}} and {\bf Lemma~\ref{pr_e}}, one has
  \begin{equation*}
    \|U^n - U^n_{(K)}\|_H \le \|r^n_{(K)}\|_H +  \|e^n_{(K)}\|_H \leq \frac{(b-a)C_\rho (2^{\ell+2}\tau)^n\tau}{K^\ell} + \frac{2\tau(b-a)C_\rho C'_\ell(n)}{K^\ell} := \frac{(b-a)C_\rho C(\ell,n)}{K^\ell},
  \end{equation*}
  which completes the proof.
\end{proof}
\begin{remark}
  The constant $C(\ell, n) = O\big(2^{(\ell+1)n}\big) = O\bigg(2^{\frac{(\ell+1)T}{\Dt}}\bigg)$. This implies a spectral convergence in gPC order $K$ for every fixed $\Dt$.
\end{remark}

\subsubsection{An error estimate of the discrete gPC method}

Now we are ready to prove the main result of the error estimate. Part of this 
estimate uses the error estimate uses the result of Jin and Qi~\cite{Qi:2013byba} for the deterministic problem.

\begin{lemma}\label{lemma}
  Let $u_0(x,z)$ be a function of bounded variation for every fixed $z$. Then the immersed interface upwind difference scheme (\ref{dis_scheme}), under the CFL condition $0 < \lambda^{\pm}(z) < 1$, has the following $\ell^1$-error bound:
  \begin{equation}
    \|U^n(z) - u(\cdot,t^n,z)\|_{\ell^1(D)}\leq C_1(z)\Gamma(c^-(z)) + C_2(z)\Gamma(c^+(z)), \quad\text{for every fixed $z$},
  \end{equation}
  where
  \begin{equation}
  \Gamma(c^{\pm}(z)) = 2\sqrt{c^{\pm}(z)\Dx(1-c^{\pm}(z)\frac{\Dt}{\Dx})t_n} + \Dx,
  \end{equation}
  and $C_1(z)$, $C_2(z)$ are bounded functions with respect to $z$.
\end{lemma}
\begin{proof}
  For every fixed $z$, this is a deterministic problem thus one can use 
 Theorem~1 in~\cite{Qi:2013byba}. Note that we have assumed $c^{\pm}(z)$ are strictly positive and bounded function with respect to $z$, so one can get a bounded $C_1(z)$ and $C_2(z)$.
\end{proof}

Next we will prove the following error estimate:
\begin{theo}
  Under Assumption~\ref{assu} and assume $u_0(x, z)$ is a function of bounded variation for every $z$. Then the following error estimate of the  discrete gPC method holds:
  \begin{equation}
    \|U^n_{(K)} - u(\cdot, t^n, \cdot)\|_H \leq C(T)(\sqrt{\Dx + \Dt} + \Dx) + \frac{(b-a)C_\rho C(\ell,n)}{K^\ell}, \quad\forall l\in\N,
  \end{equation}
  where $C(T)$ depends only on time $T$ and $C(\ell,n)$ depends on $\Dt$ and $\ell$.
\end{theo}
\begin{proof}
  First we split the error into two parts:
  \begin{equation}
    \|U^n_{(K)} - u(\cdot, t^n, \cdot)\|_H \leq \|U^n - u(x_i, t^n, z)\|_H + \|U^n_{(K)} - U^n\|_H.
  \end{equation}
  For the first part, it is the error of numerical scheme~(\ref{dis_scheme}), using {\bf Lemma~\ref{lemma}} one gets
  \begin{equation}
    \begin{split}
      \|U^n_{(K)} - u(\cdot, t^n, \cdot)\|_H &= \Bigg(\int\|U^{n}(z) - u(\cdot, t^n, z)\|_{l^1(D)}^2\rho(z)\diff z\Bigg)^{1/2} \\
      & \leq \Bigg(\int (C_1(z)\Gamma(c^-(z)) + C_2(z)\Gamma(c^+(z)))^2\rho(z)\diff z\Bigg)^{1/2} \\
      & \leq \Bigg(\int [(C_1(z)\Gamma(c^-(z))]^2\rho(z)\diff z\Bigg)^{1/2} + \Bigg(\int [(C_2(z)\Gamma(c^+(z))]^2\rho(z)\diff z\Bigg)^{1/2}.
    \end{split}
  \end{equation}
  The last inequality is obtained by Minkowski inequality. Notice that $C_1(z)$ is bounded and
  \begin{equation}
    \begin{split}
      \Bigg(\int [\Gamma(c^{\pm}(z))]^2\rho(z)\diff z\Bigg)^{1/2} & \leq 2\Bigg(\int c^{\pm}(z)\Dx(1-c^{\pm}(z)\frac{\Dt}{\Dx})t_n\rho(z)\,dz\Bigg)^{1/2} + \Bigg(\int \Dx^2\rho(z)\diff z\Bigg)^{1/2} \\
      &= 2\Bigg(t_n\Dx\int c^{\pm}(z)\rho(z)\diff z - t_n\Dt\int (c^{\pm}(z))^2\rho(z)\diff z\Bigg)^{1/2} + \Dx \\
      &\leq C(T)\sqrt{\Dx + \Dt} + \Dx.
    \end{split}
  \end{equation}
Therefore one gets
  \begin{equation}\label{err1}
    \|U^{n}- u(\cdot, t^n, \cdot)\|_H \leq C(T)(\sqrt{\Dx + \Dt} + \Dx).
  \end{equation}
  For the second part, according to {\bf Theorem~\ref{gpc_conv}} we have
  \begin{equation}
    \|U^n - U^n_{(K)}\|_H \leq \frac{(b-a)C_\rho C(\ell,n)}{K^\ell}, \quad\forall\ell\in\N.
  \end{equation}
  Then by adding these two parts we complete the proof.
\end{proof}

\section{A gPC method for the Liouville equation with discontinuous potential}
In this section we study the Liouville equation in classical mechanics with random uncertainties:
\begin{equation}\label{eq_Liouville}
  u_t + v  u_x - V_x  u_v = 0, \quad t>0, \quad x, v \in \R,
\end{equation}
with initial condition 
\begin{equation}
  u(x, v, 0, z) = u_0(x, v, z),
\end{equation}
where $u(x, v, t, z)$ is the density distribution of a classical particle at position $x$, time $t$ and traveling with velocity $v$. $V(x, z)$ is the potential depending on a random variable $z$

The Liouville equation has bicharacteristics defined  by Newton's second law:
\begin{equation}\label{H_sys}
    \frac{dx}{dt}=v, \quad \frac{dv}{dt} = -V_x(x, z),
\end{equation}
which is a Hamiltonian system with the random Hamiltonian
\begin{equation}
    H = \frac{1}{2} v^2 + V(x,z).
\end{equation}

If $V(x, z)$ is discontinuous with respect to $x$ which corresponds to a random potential barrier, then the characteristic speed of the Liouville equation given by (\ref{H_sys}) is infinity at the discontinuous point and a conventional numerical scheme becomes difficult. On the other hand, it is known from classical mechanics that the Hamiltonian remains constant across a potential barrier. Based on this principle, Jin and Wen proposed a framework, called Hamiltonian preserving scheme in which they build the interface condition into the scheme according to the behavior of a particle across the potential barrier~\cite{Wen:2005ueba}, \cite{Novak:2006cpba}.

As in the previous section, we first discretize equation (\ref{eq_Liouville}) using the Hamiltonian preserving scheme in which we regard the random variable $z$ as a fixed parameter.

Without loss of generality, we employ a uniform mesh with grid points at $x_{i+1/2}$, $i=0,\dotsc, N$ in the $x$-direction and $v_{j+1/2}$, $j = 0, \dotsc, M$ in the $v$-direction. The cells are centered at $(x_i,v_j)$, $i = 1, \dotsc, N$, $j = 1, \dotsc, M$ with $x_i = (x_{i+1/2} + x_{i-1/2})/2$ and $v_j = (v_{j+1/2} + v_{j-1/2})/2$. The mesh size is denoted by $\Delta x=x_{i+1/2}-x_{i-1/2}$ and $\Delta v=v_{i+1/2}-v_{i-1/2}$. Also we assume that the discontinuous points of potential $V$ are located at some grid points. Let the left and right limits of $V$ at point $x_{i+1/2}$ be $V^+_{i+1/2}$ and $V^-_{i+1/2}$ respectively. The scheme reads:
\begin{equation}\label{semi_dis}
  \partial_t u_{ij}(z) + v_j\frac{u^-_{i+1/2,j}(z)-u^+_{i-1/2,j}(z)}{\Dx}-\mathrm{D}V_i(z)\frac{u_{i,j+1/2}(z)-u_{i,j-1/2}(z)}{\Dv}=0,
\end{equation}
here 
\begin{equation}
  \mathrm{D}V_i(z) := \dfrac{V^-_{i+1/2}(z)-V^+_{i-1/2}(z)}{\Dx}.
\end{equation}
We also need to determine the numerical fluxes $u_{i, j+1/2}(z)$ and $u^{\pm}_{i+1/2, j}(z)$ at each cell interface. 

\subsection{A first order finite difference approximation}
Here we can use the standard first order upwind scheme for the fluxes $u^{\pm}_{i+1/2, j}(z)$ since the wave speed in this direction, which is $v_j$, has nothing to do with the random variable $z$. So the characteristic in fact is deterministic. For example we consider the case $v_j>0$, according to the Hamilton preserving scheme such fluxes read,
\begin{equation}\label{cck}
  \begin{aligned}
    &u^-_{i+1/2,j}(z) = u_{ij}(z), \\
    &u^+_{i+1/2,j}(z) = 
    \begin{cases}
      c_1 u_{i+1,k}(z) + c_2 u_{i+1,k+1}(z), &\text{when transmission}, \\
      u_{i+1, k}(z), &\text{when reflection},
    \end{cases}
  \end{aligned}
\end{equation}
Here $k$ is an index determined by the energy conservation across the interface (see~\cite{Wen:2005ueba}),
\begin{equation}\label{hpc}
  \frac{1}{2}(v_j)^2 + V^-_{i+1/2} = \frac{1}{2}(v^+)^2 + V^+_{i+1/2},
\end{equation}
where $v^+$ is the velocity across the barrier. If $(v_j)^2 + 2(V^-_{j+1/2} - V^+_{j+1/2})>0$, particle will transmit, thus
\begin{equation}
  v^+ = \sqrt{(v_j)^2 + 2(V^-_{j+1/2} - V^+_{j+1/2})},
\end{equation}
$k$ is the index such that
\begin{equation}
  v_k\leq v^+ < v_{k+1},
\end{equation}
and $c_1$ and $c_2$ are the coefficients of a linear interpolation,
\begin{equation}
  c_1 = \frac{v^+ - v_k}{\Dv}, \quad c_2 = \frac{v_{k+1} - v^+}{\Dv}, \quad c_1 + c_2 = 1.
\end{equation}
If $(v_j)^2 + 2(V^-_{j+1/2} - V^+_{j+1/2})<0$, then particle will reflect,
$k$  is the index such that
\begin{equation}\label{-j}
  v_k = -v_j.
\end{equation}
When $v_j<0$, we can determine $c_1$, $c_2$ and $k$ similarly using the Hamilton preserving condition~(\ref{hpc}).

For the flux $u_{i, j+1/2}(z)$, one should be careful when dealing with it. Unlike the flux in $x$-direction, the wave speed in $v$-direction, $\mathrm{D} V_i(z)$, depends on the random variable $z$ such that the characteristic is random. This will make the discontinuity of solution in the physical space propagate into the random space if we use a characteristic dependent scheme (\ie upwind scheme), which results a bad regularity of the solution with respect to $z$. 

Here we use the Lax-Friedrichs flux which is a characteristic independent scheme for the $v$-direction flux in general:
\begin{equation}\label{v-direction}
  u_{i,j+1/2}(z) = \frac{1}{2}\left[\frac{\alpha}{DV_i(z)}\big(u_{i,j+1}(z) - u_{ij}(z)\big) - \big(u_{ij}(z) + u_{i,j+1}(z)\big)\right],
\end{equation}
where $\alpha$ is a constant satisfying $\alpha\geq \max\limits_{i, z}|DV_i(z)|$.

From the discussion above, one can easily see that the fluxes of $x$-direction and $v$-direction are both smooth functions with respect to $z$. As in Section 2, we can conclude that the solution of this scheme, which is $u_{ij}(z)$, are smooth functions of $z$ for each $i, j$. Then we apply the standard gPC-SG method to this discrete system, same as in Section 2, one can expect a fast convergence of gPC expansion to the discretized solution when the mesh size $\Dx$ and time step $\Dt$ are fixed. The justification is the same as in Section 2.

\begin{remark}
Here any central schemes, such as the local Lax-Friedrichs scheme, can be used
besides the Lax-Friedrichs scheme.
\end{remark}
\subsection{A formally second order  spatial discretization}
In previous two sections, we have presented our discrete gPC scheme with a first order spatial discretization. In the following, we will give the formally second 
order spatial discretization. Specifically, the spatial numerical flux
used in the Hamilton Preserving scheme~\cite{Wen:2005ueba} is given by (consider the case when $v_j>0$)
\begin{equation}\label{xflux}
  \begin{aligned}
    &u^-_{i+1/2,j}(z) = u_{ij}(z) + \frac{\Delta x}{2}s_{ij}(z), \\
    &u^+_{i+1/2,j}(z) = 
    \begin{cases}
      c_1\left(u_{i,k}(z) + \dfrac{\Delta x}{2}s_{i,k}(z)\right) + c_2\left(u_{i,k+1}(z) + \dfrac{\Delta x}{2}s_{i,k+1}(z)\right),  \\
      u_{i+1, k}(z) - \dfrac{\Delta x}{2}s_{i+1, k}(z),
    \end{cases}
  \end{aligned}
\end{equation}
where $s_{ij}$ is the numerical slope, $c_1,c_2,k$ are determined by the Hamilton preserving scheme just as the first order case in previous subsection~(\ref{cck})--(\ref{-j}).

Since the solution contains discontinuities, a second order scheme will necessarily introduce numerical oscillations. In order to suppress these oscillations, one can use the limited slope, in the spirit of total-variation-diminishing
(TVD) framework~\cite{LeVeque}. Most of the slope limiteds used in the
shock capturing community are non-smooth functions, while in our approach
the regularity in $z$ is essential. To this aim, we use smooth slope
limiters called BAP, introduced in~\cite{Liu:1998faba}. For the backward and forward differences, 
\begin{equation}
\begin{aligned}
  s_l(z) &= (u_{ij}(z)-u_{i-1,j}(z))/\Delta x, \\
  s_r(z) &= (u_{i+1,j}(z)-u_{ij}(z))/\Delta x,
\end{aligned}
\end{equation}
at $(x_i,v_j)$, the BAP slope is given by
\begin{equation}
  s_{ij}(z)=\mathcal{B}^{-1}\left(\frac{\mathcal{B}(s_l(z))+\mathcal{B}(s_r(z))}{2}\right).
\end{equation}
Some examples of smooth $\mathcal{B}(x)$ include
\begin{equation}
  \begin{aligned}
    \mathcal{B}(x)&=\arctan(x), \quad&\mathcal{B}^{-1}(x)&=\tan(x), \\
    \mathcal{B}(x)&=\tanh(x), \quad& \mathcal{B}^{-1}(x)&=\tanh^{-1}(x), \\
    \mathcal{B}(x)&=\frac{x}{\sqrt{1+x^2}}, \quad& \mathcal{B}^{-1}(x)&=\frac{x}{\sqrt{1-x^2}}.
  \end{aligned}
\end{equation}

\subsection{The full discretization}

Next we need to define the numerical flux in the $v$-direction. To this stage, in order to get a smooth discrete solution (with respect to $z$), we also need to choose some scheme that does not depend on characteristic information.
Here we  use the Lax-Wendroff scheme. The $v$-direction flux:
\begin{equation} \label{vflux}
u_{i,j+1/2}(z) = \frac{1}{2}(u_{i,j+1}(z)+u_{i,j}(z))
  + (\mathrm{D}V_i(z))\frac{\Delta t}{2\Delta v}(u_{i,j+1}(z)-u_{ij}(z)).
\end{equation}

So combine~(\ref{semi_dis}), (\ref{xflux}) and (\ref{vflux}), we get a second order in space and velocity,  whose solution is smooth with respect to $z$,
 written as
\begin{equation}
  \dt u_{ij}(z) = \mathrm{RHS}(z).
\end{equation}

We now the gPC-Galerkin method to this discrete system, for the $k$-th component $u^{n,k}_{ij}$ in gPC expansion we have:
\begin{equation} \label{gPC2}
  \dt u^{k}_{ij}(z) = \left<\mathrm{RHS}(z),P_k(z)\right>,
\end{equation}
where $P_k(z)$ is the $k$-th order orthogonal polynomial and $\left<\cdot\right>$ is the inner product on the random space.
Due to the complicated nonlinear form of RHS($z$), we will use numerical integration, \ie, the Gauss-quadrature to calculate the right hand side of~(\ref{gPC2}).
\begin{equation}\label{RHS-col}
  \left<\mathrm{RHS}(z),P_k(z)\right> = \sum_{m=0}^{M}\mathrm{RHS}(z_m)P_k(z_m)w_m,
\end{equation}
where $M$ is the total number of quadrature points we choose and $z_i$, $w_i$ are the Gauss quadrature points and corresponding weights. 
Here we summarize the algorithm on every time step:
\begin{itemize}
  \item First, use the gPC expansion $u^n_{ij}(z) = \sum\limits_{k=0}^K u^{n,k}_{i,j}P_k(z)$ to compute $u^n_{ij}(z_m)$. Notice that one only needs to compute $P_k(z_m)$, which is independent of time thus can be pre-computed before time marching. 
  \item Using $u^n_{ij}(z_m)$ and (\ref{xflux}), (\ref{vflux}) to get RHS($z_m$) for every $i,j,m$.
  \item Finally by (\ref{RHS-col}) and time marching using the forward Euler or Runge-Kutta methods to get $u^{n+1,k}_{ij}$ for every $i,j,k$. This finishes one time step. 
\end{itemize}
 
\begin{remark}
  For the convection equation (\ref{eq:convection}), one can simply replace $U_i^n(z)$ by $U_i^n(z) + s_i(z)\Dx/2$ in (\ref{dis_scheme}) and follow the procedure above to get a second order scheme in spatial domain.
\end{remark}
\section{Numerical examples}

In this section we will conduct some numerical experiments to show the performance of the proposed methods and check their numerical accuracy.

\subsection{Example 1: the scalar convection equation with discontinuous coefficient}
We consider the initial problem
\begin{equation}
  \left\{
  \begin{aligned}
    &u_t + \big[(c(x,z)u\big]_x=0, &\quad& t>0, x\in\R, \\
    &u(x,0) = u_0(x), &&x\in\R,
  \end{aligned}
  \right.
\end{equation}
with
\begin{equation}
    c(x,z) = 0.3z + 
    \begin{cases}
        c^- > 0, \quad &\text{if $x < 0$},\\
        c^+ > 0, \quad &\text{if $x > 0$},
    \end{cases}
\end{equation}
where $z$ is uniformly distributed on $[-1,1]$ (thus the gPC basis should be the normalized Legendre polynomials) and we treat the random variable $z$ as a small perturbation such that $(c^{\pm} + 0.3z) > 0$ for any $z\in[-1,1]$. 

In this example, we set the initial data as
\begin{equation}
  u_0(x) = \cos (0.25\pi x), \quad \text{on $[-1,3]$},
\end{equation}
and an interface is located at $x=0$ with the condition:
\begin{equation}
  u(0^+, t, z) = \rho(z) u(0^-, t, z),
\end{equation}
where 
\begin{equation}
  \rho(z) = \frac{c^- + 0.3z}{c^+ + 0.3z},
\end{equation}
for the conservation of flux.

The analytic solution of this simple model problem can be easily obtained by using the method of characteristic including the interface condition~\cite{Wen:2005ueba}:
\begin{equation}\label{ana_sol}
  u(x, t, z)=   
  \begin{cases}
    u_0\big(x - (c^+ + 0.3z)t\big), \quad &x > (c^+ + 0.3z)t,\\
    \rho(z) u_0\big(\rho(z)[x - (c^+ + 0.3z)t]\big), &0 < x < (c^+ + 0.3z)t,\\
    u_0\big(x - (c^- + 0.3z)t\big),  &x < 0.
  \end{cases}
\end{equation}
In the following examples, we set $c^- = 1$, $c^+ = 2$, and final time $T = 1$. The expectation and variance of the analytic solution can be obtained using~(\ref{exp}) and~(\ref{var}).

For numerical solutions, we compute their expectation by 
\begin{equation*}
  \mathbb{E}_i(t^n):=\mathbb{E}[u(x_i,t^n,z)]=\int u(x_i,t^n,z)\rho(z)\diff z = \hat{U}^n_{i,(0)}, 
\end{equation*}
and their variance by
\begin{equation*}
  \mathbb{V}_i(t^n):=\mathbb{E}\big[(u-\mathbb{E}(u))^2\big] = \sum_{k=1}^K \big(\hat{U}^n_{i,(k)}\big)^2.
\end{equation*}
The norm for measuring the error between the analytic solution and the numerical solution is~$\ell^1$.

\begin{figure}[htbp]
  \centering
  \includegraphics[width=0.8\textwidth]{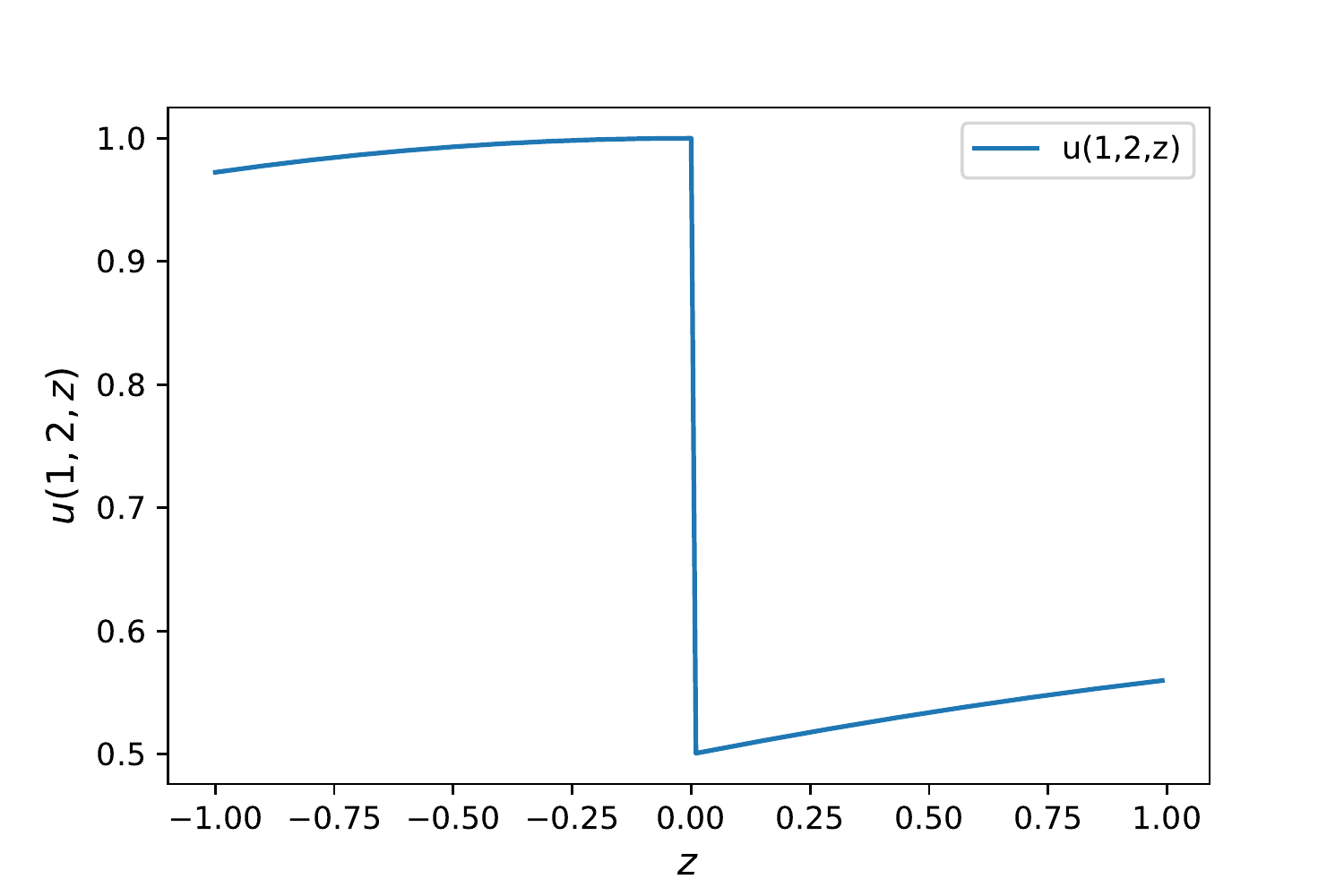}
  \caption{Example 1. The analytic solution~(\ref{ana_sol}) at $t=1$ and $x=2$ is a discontinuous function of $z$.}
  \label{1}
\end{figure}
Figure~\ref{1} shows that the analytic solution~(\ref{ana_sol}) has a discontinuity at $z = 0$ when $x = 2$.  Figure~\ref{2} shows that the expectation and variance of the analytic solution. In this case, one can expect a low  convergence rate of the standard gPC-SG method.  
\begin{figure}[htbp]
  \centering
  \includegraphics[width=0.8\textwidth]{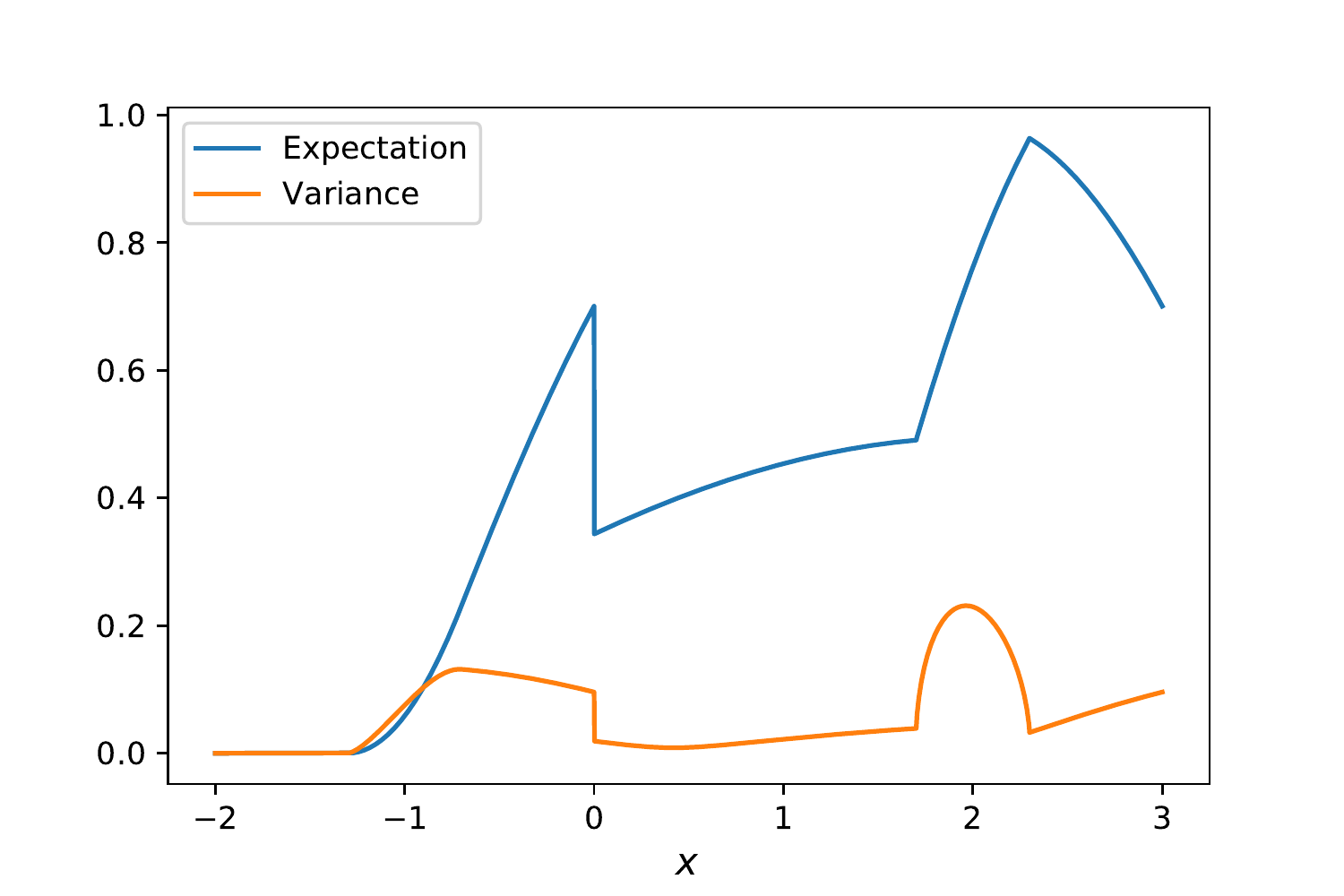}
  \caption{Example 1. The expectation and variance of the analytic solution.}
  \label{2}
\end{figure}

\subsubsection{The first order finite difference approximation}
In this subsection, we will give the numerical results of our discrete gPC-SG method.  Figure~\ref{3} shows the numerical expectation and variance compared with the analytic solution with $\Dx = 0.001$, $\Dt = \dfrac{1}{4}\Dx$ and gPC order $K=20$. The discrepancy on the variance is due to the poor resolution of the
first order spatial discretization, which is improved with the second order
spatial discretization to be used later.

Next we conduct the convergence test only for the gPC approximation.
We fix $\Dx=0.005$ and $\Dt = \dfrac{1}{5}\Dx$ in all computations with
different $K$. Figure~\ref{4} shows that the $\ell^1$ error decays very fast with respect to the gPC order $K$. When $K=4$, it decays to the numerical error of the finite difference method. 
\begin{figure}[htbp]
  \includegraphics[width=0.5\textwidth]{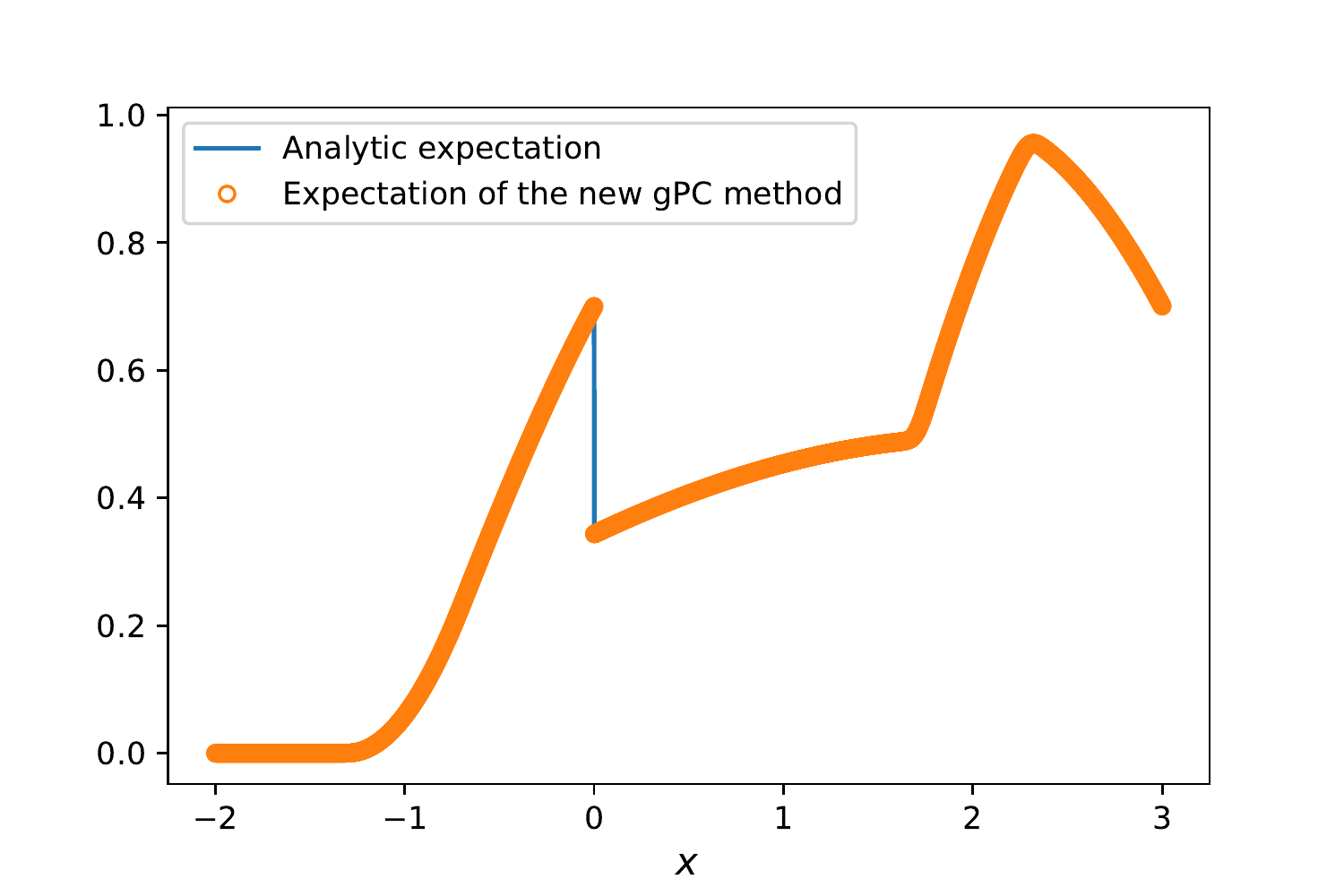}
  \includegraphics[width=0.5\textwidth]{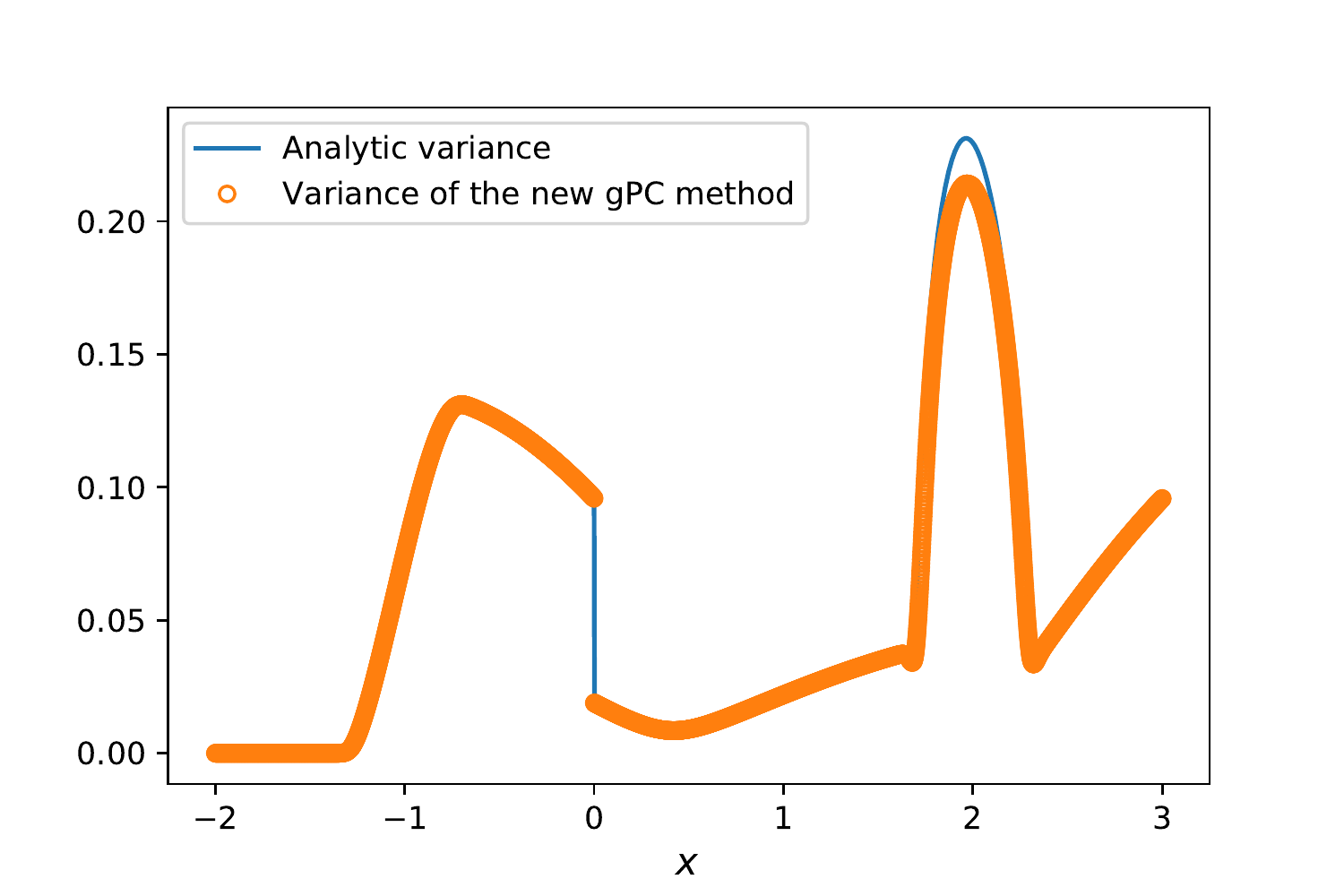}
  \caption{Example 1. The analytic solution compared with the new gPC-SG method using first order finite difference approximation with $\Dx = 0.001$, $\Dt = \dfrac{1}{4}\Delta x$, gPC order $K=20$.}
  \label{3}
\end{figure}

\begin{figure}[htbp]
  \centering
  \includegraphics[width=0.8\textwidth]{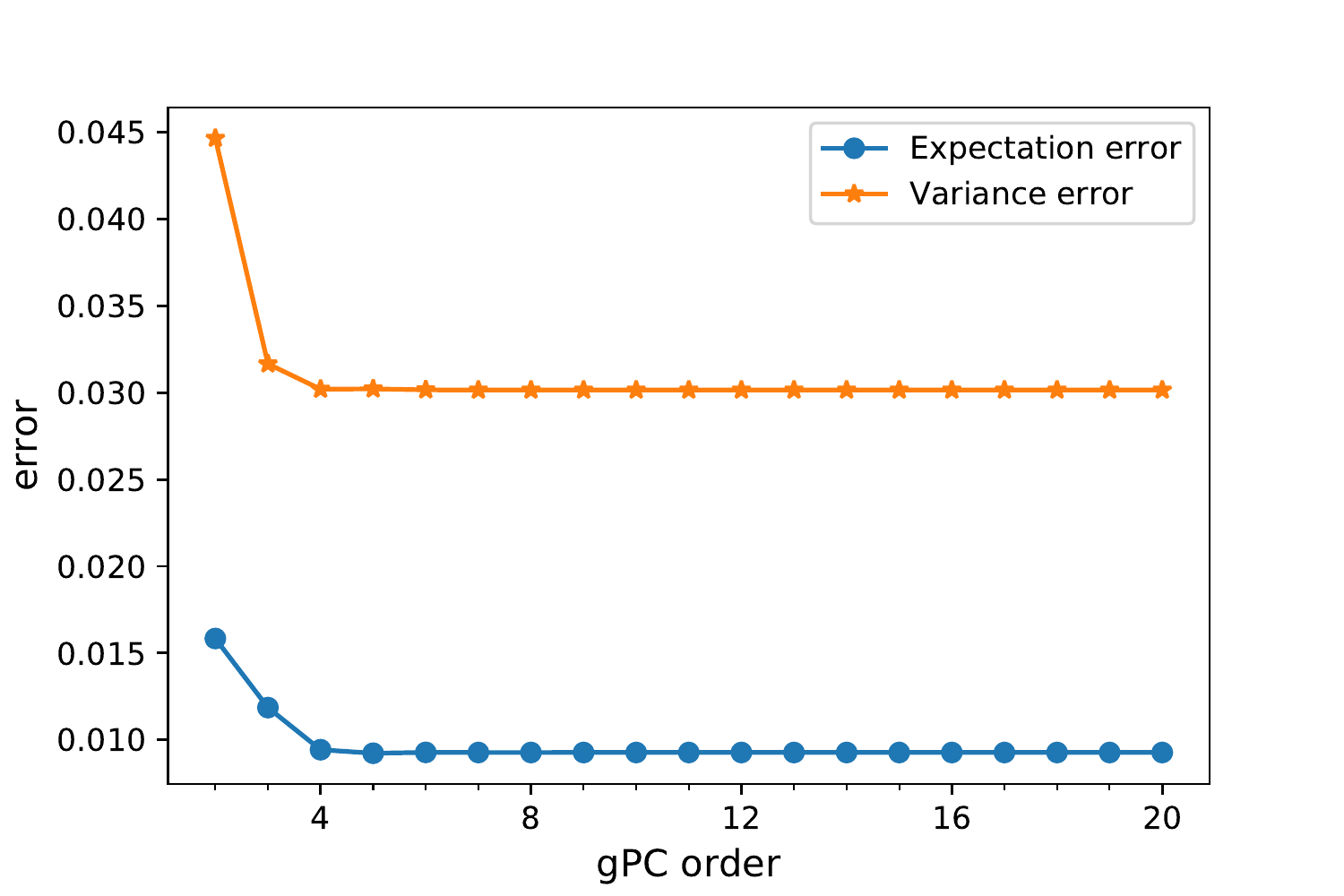}
  \caption{Example 1. The first order finite difference approximation with $\Dx=0.005$, $\Dt=\dfrac{1}{5}\Dx$: the $\ell^1$ error versus the gPC order.}
  \label{4}
\end{figure}

However, in Figure~\ref{4}, since the finite difference error dominates the
gPC error, it is difficult to verify the convergence rate of the gPC method. In order to examine the gPC error, we fix $\Dx$ and $\Dt$, and
compare the numerical solutions with different $K$, with the case of $K=30$ serving as the reference solution.  We measure the $\ell^1$ error between each $K=2,3,\dotsc,20$ and $K=30$. The result is shown 
in Figure~\ref{5}, in which an exponential convergence in the gPC approximation can be observed by using the $\log$-$\log$ plot. Note that if the convergence order is algebraic, the curve should be a line.
Here the curve shape shows the  exponential decay of the gPC error.

\begin{figure}[htbp]
  \centering
  \includegraphics[width=0.8\textwidth]{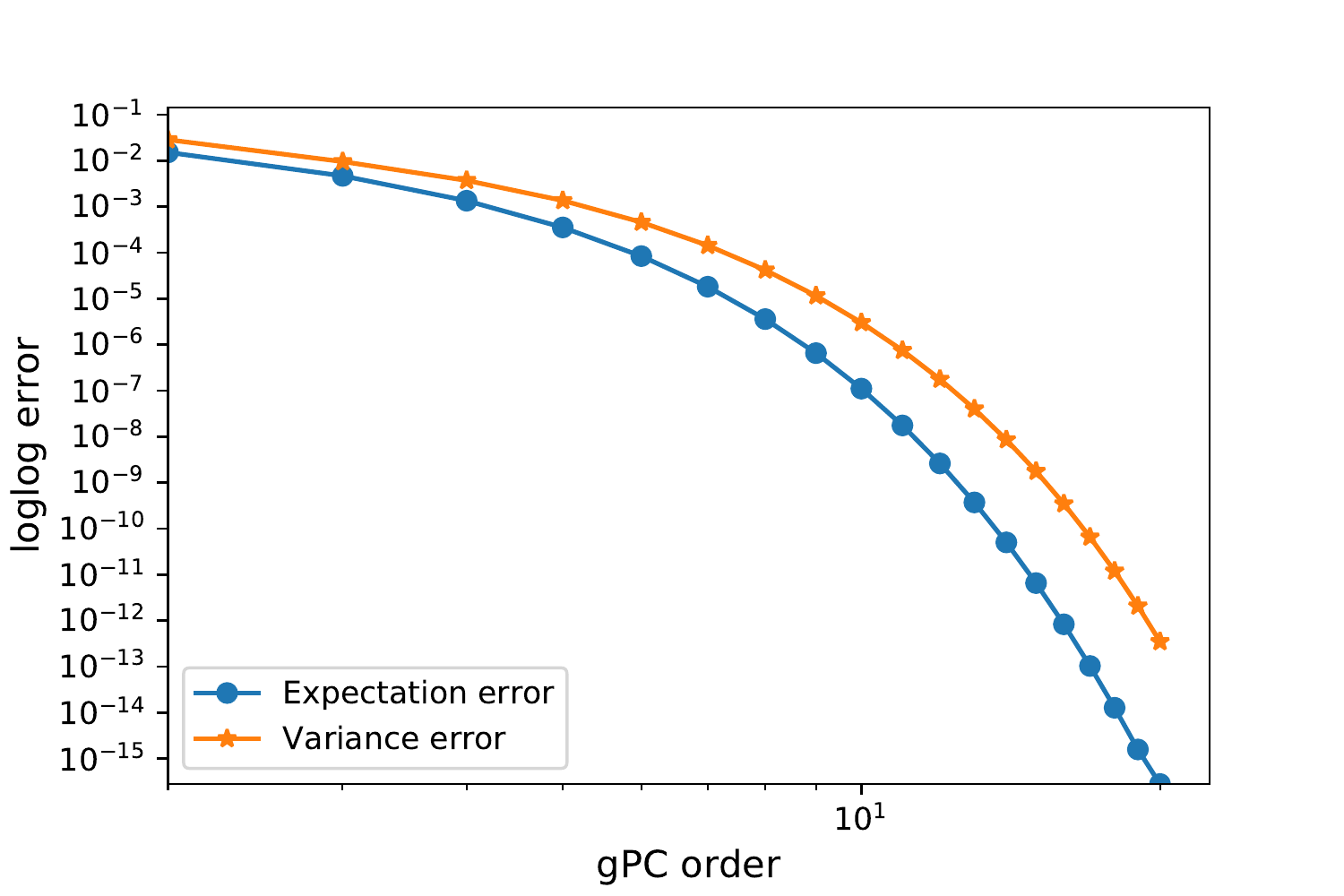}
  \caption{Example 1. The first order finite difference approximation $\Dx=0.005$, $\Dt=\dfrac{1}{5}\Dx$: the gPC error versus the gPC order by a log-log plot
(with other numerical parameters fixed).}
  \label{5}
\end{figure}

%

\subsubsection{The second order finite difference approximation}
For the second order scheme, we use the same set up as in the first order case. Figure~\ref{mv2} shows the expectation and variance compared with the analytic solution,
which gives a more accurate solution than the first
order approximation especially for the variance around $x=2$. 

Figure~\ref{mverr2} and Figure~\ref{mverrlog2} show the convergence of the numerical method in the gPC order from which one can observe the fast convergence. Comparing Figure 7 with Figure 4, we can see that the second scheme has a better total $\ell^1$ error. But the rate of the gPC convergence shown in Figure~\ref{mverrlog2} is  not as fast as the first order scheme. This is hardly surprising since our spectral convergence
depends on the smoothness of the discrete solutions, and the smoothness is
given by the numerical viscosity which is larger 
for the first order spatial discretization. The second order spatial
discretization offers better accuracy away from the discontinuities and better
resolutions at discontinuities, but  because it is closer
to the analytic solution (which is not smooth) thus less smooth than the first order one, and smoothness of the discrete solution is what 
 our spectral convergence relies upon, thus its gPC congerence rate, compared with
the first order one, should be slower.  However this does not mean that 
the second order method is inferior to the first one, since one has to 
consider the {\it overall} error, including the contributions of error from
the spatial discretization in this problem. By comparing  Figure 6 with Figure 3, and 
Figure 7 with Figure 4, it is obvious that the second order scheme outperforms
the first order one.
 
\begin{figure}[htbp]
  \includegraphics[width=0.5\textwidth]{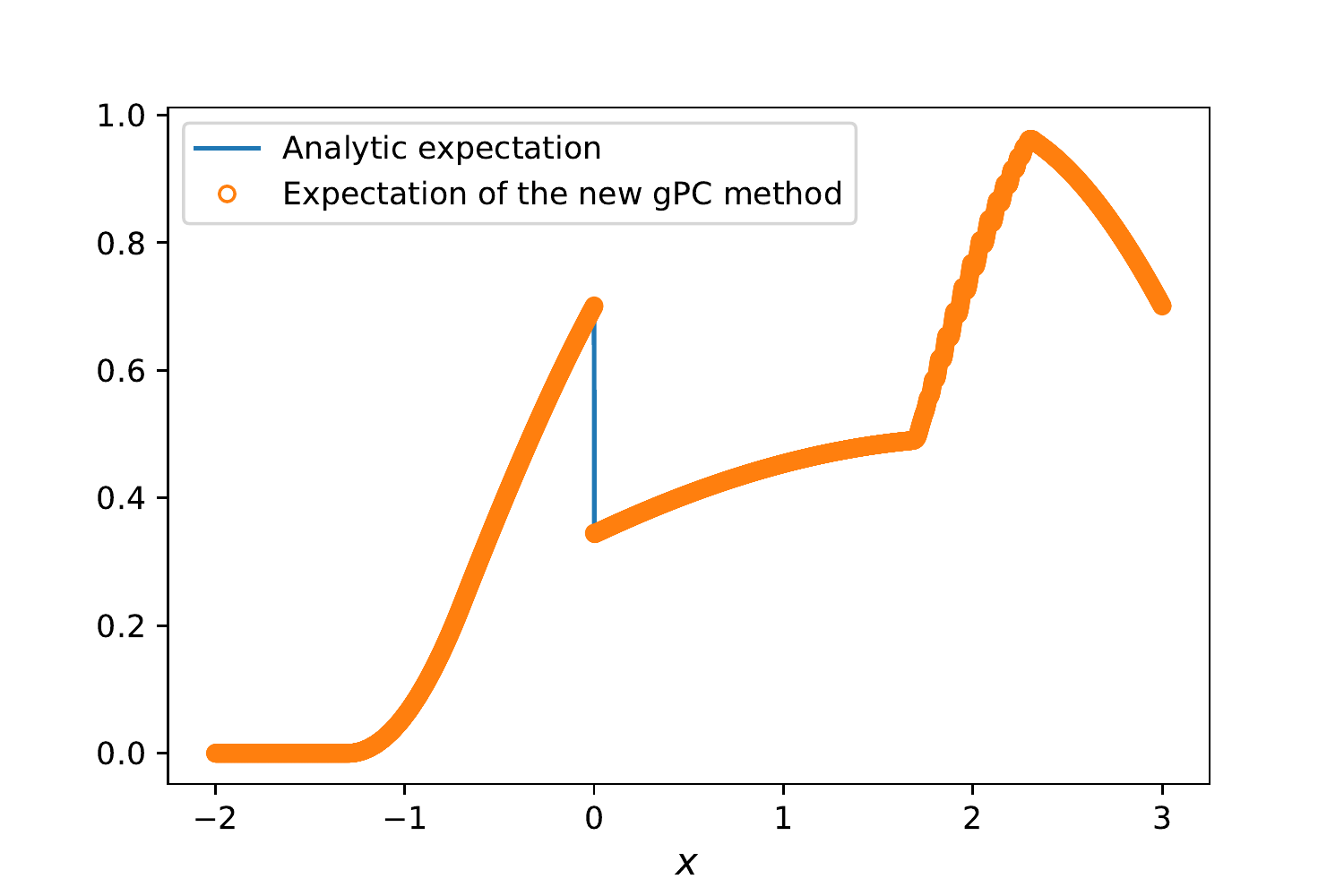}
  \includegraphics[width=0.5\textwidth]{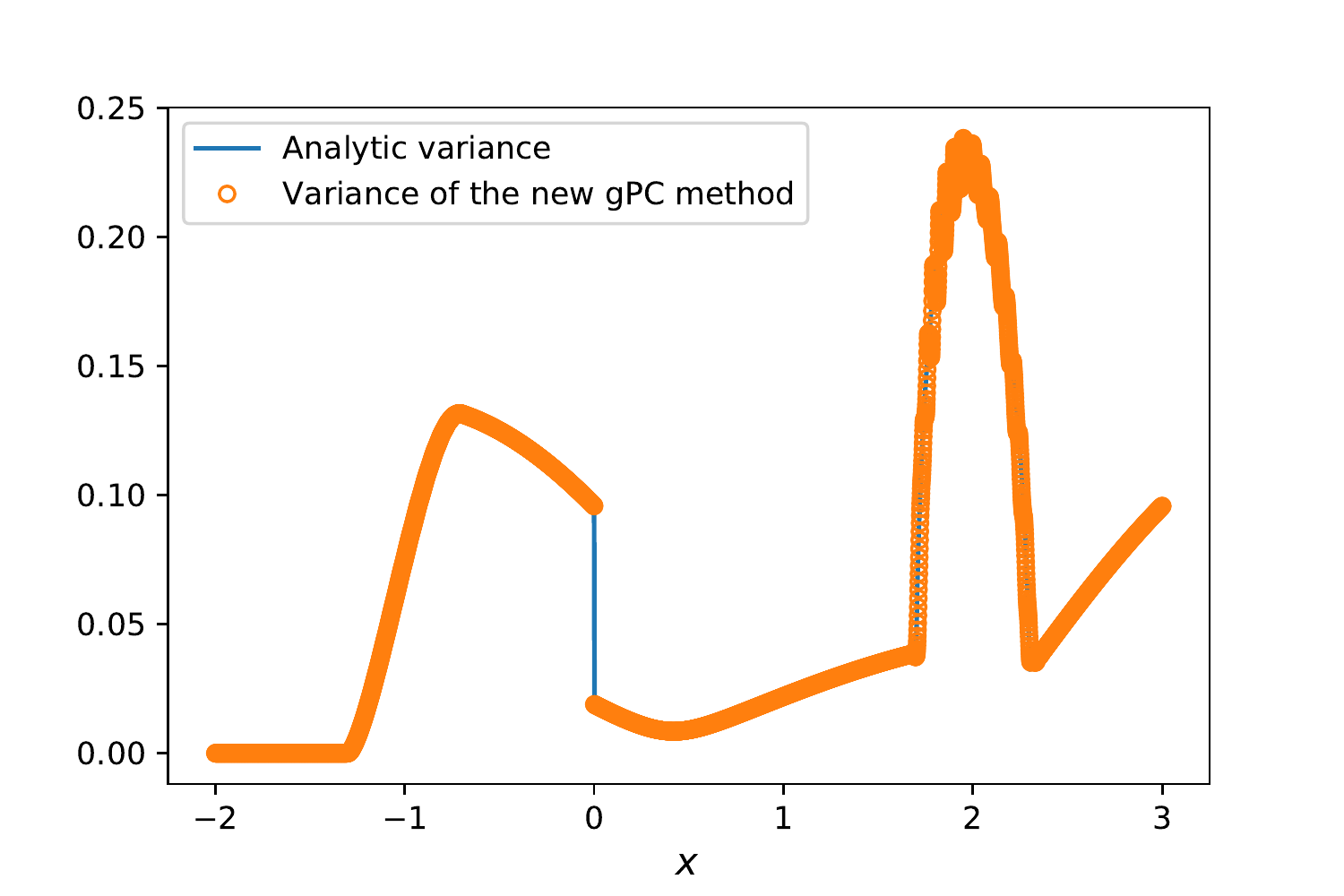}
  \caption{Example 1. The analytic solution compared with the new gPC-SG method using the second order finite difference approximation with  $\Delta x = 0.001$, $\Delta t = \dfrac{1}{4}\Delta x$, gPC order $K=20$.}
  \label{mv2}
\end{figure}

\begin{figure}[htbp]
  \centering
  \includegraphics[width=0.8\textwidth]{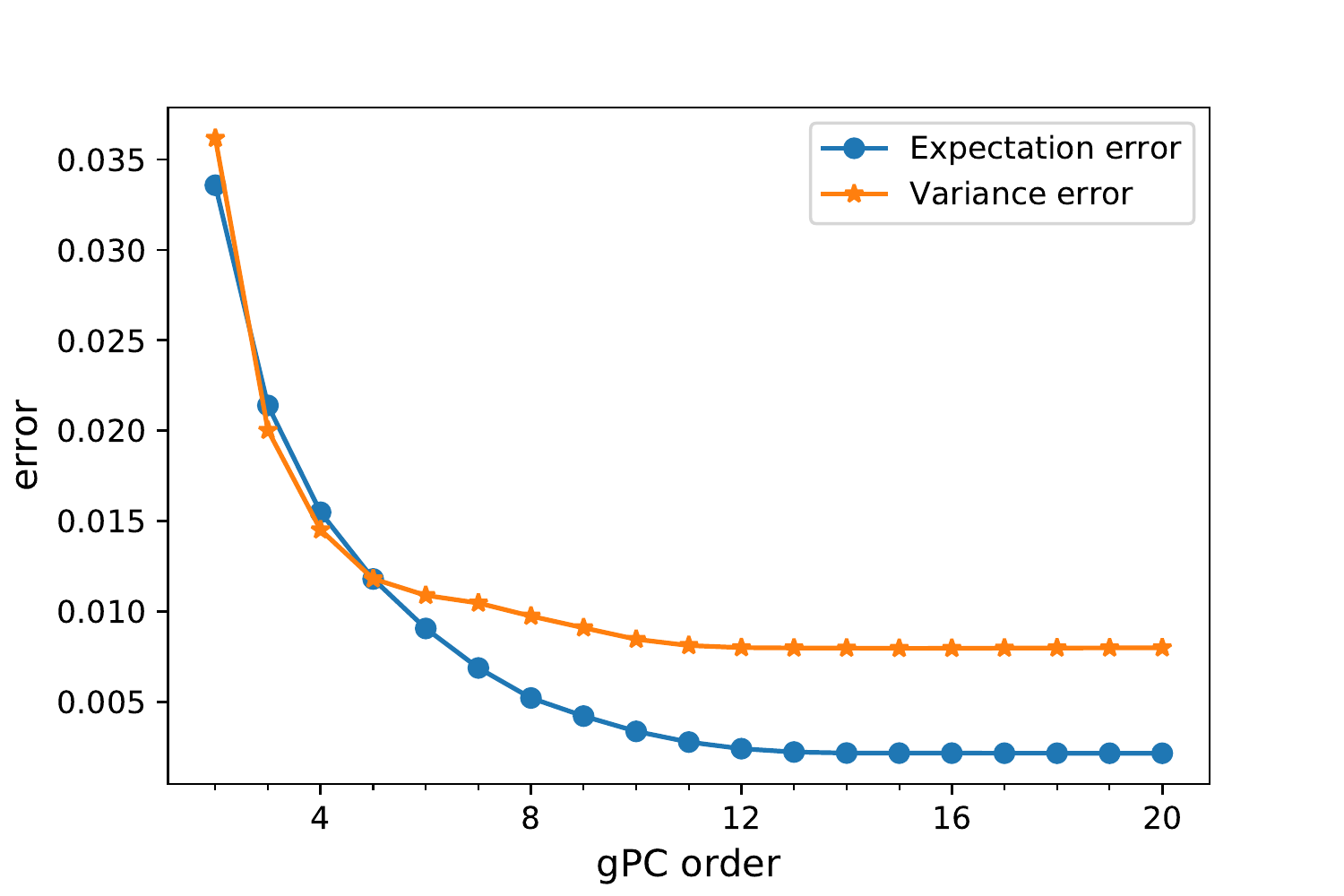}
  \caption{Example 1. The second order finite difference approximation $\Dx=0.005$, $\Dt=\dfrac{1}{5}\Dx$: the 
  $\ell^1$ error versus the gPC order.}
  \label{mverr2}
\end{figure}

\begin{figure}[htbp]
  \centering
  \includegraphics[width=0.8\textwidth]{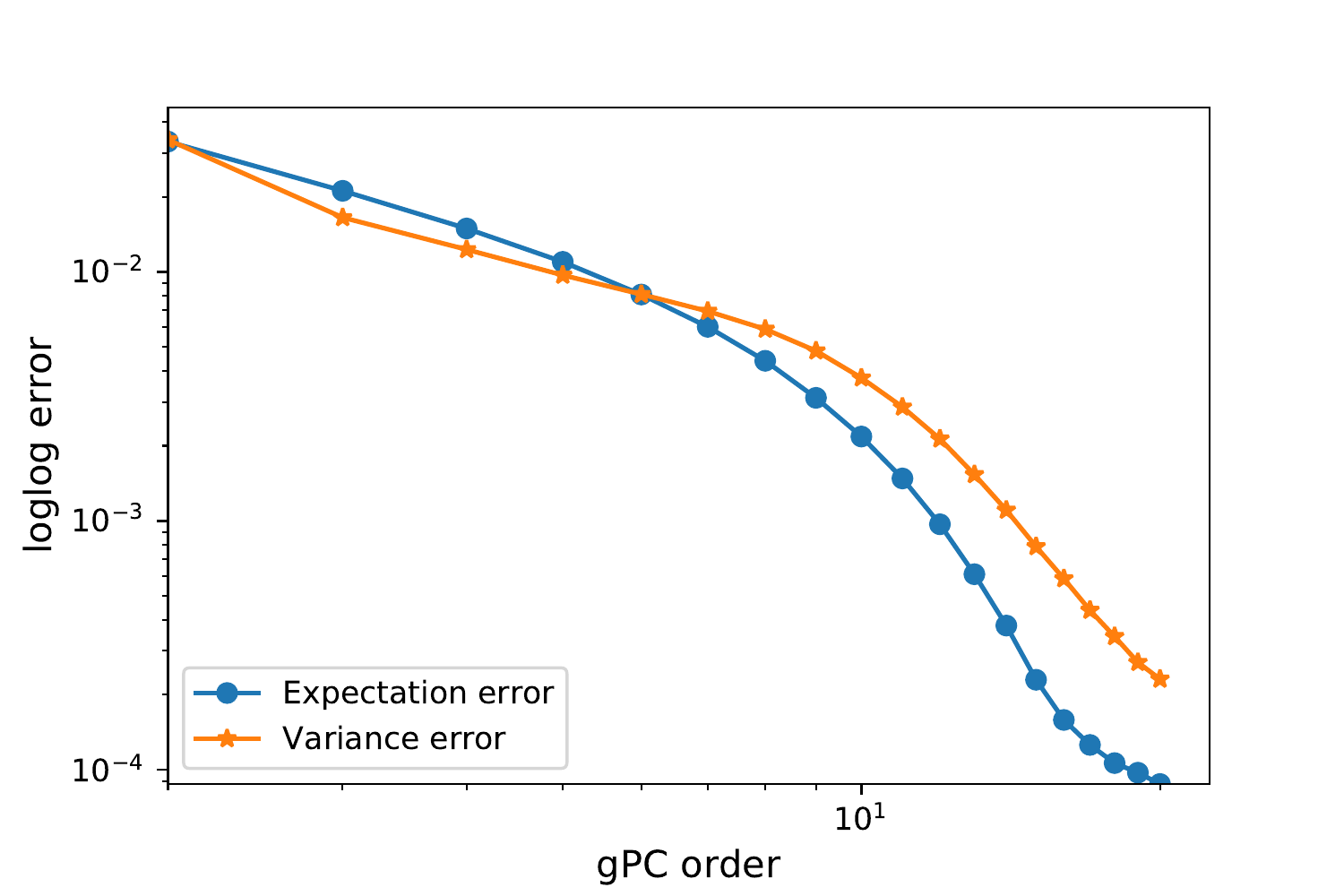}
  \caption{Example 1. The second order finite difference approximation $\Dx=0.005$, $\Dt=\dfrac{1}{5}\Dx$: the gPC error versus the gPC order by a log-log plot
(with other numerical parameters fixed).}
  \label{mverrlog2}
\end{figure}


\subsection{Example 2: the Liouville equation with a discontinuous potential}

Recall the Liouville equation
\begin{equation}
  u_t+v u_x-V_x u_v=0, \quad t>0, \quad x,v\in R,
\end{equation}
with the random potential given by
\begin{equation}
  V(x,y)=V_0(x)+0.1xz,
\end{equation}
where $z$ is uniformly distributed on $(-1,1)$ and 
\begin{equation}
  V_0(x)=
  \begin{cases}
    0.2, \quad &x<0, \\
    0, \quad &x>0.
  \end{cases}
\end{equation}

For the given initial data, one cannot get an analytic solution for this problem. Instead we will use the {\it collocation method} as a comparison. In collocation method, one solves the Liouville equation~(\ref{eq_Liou}) at a discrete set of $\{z_i\}_{1\leq i\leq M}$ called sample points in the corresponding random space. For every fixed $z_i$, we only need to solve a {\it deterministic} Liouville equation with discontinuous potential using Hamilton preserving scheme~\cite{Wen:2005ueba}. Then the expectation and variance can be obtained by the quadrature rules of~(\ref{exp}) and~(\ref{var}). In the following examples, we choose $\{z_i\}_{1\leq i\leq M}$ as the roots of $M$th order Legendre polynomials and use the Gauss-Legendre quadrature to obtain the expectation and variance.

For the gPC method we need to evaluate $\int_{-1}^1 V_0(z)P_j(z)P_k(z) \rho(z)\diff z$, which, for this simple case, is given by
\begin{equation}
  \int_{-1}^1 V_0(z)P_j(z)P_k(z) \rho(z)\diff z =
  \begin{cases}
    \dfrac{j+1}{\sqrt{(2j+1)(2j+3)}}, &k=j+1, \\
    V'_0(x), &k=j, \\
    \dfrac{j}{\sqrt{4j^2-1}}, &k=j-1.
  \end{cases}
\end{equation}
Here one has a symmetric tridiagonal matrix.

As an illustration of the singularity of the solution caused by the discontinuous potential , we use a continuous initial data:
\begin{equation}
  u(x,v,0)=
  \begin{cases}
  \sin[2\pi(0.25-(x^2+v^2))], &x^2+v^2<0.25, \\
  0, &\text{otherwise}.
  \end{cases}
\label{ex2-init2}
\end{equation}
The expectation of the solution by using the collocation method with $M=20$ sample points and our new gPC-SG method with gPC order $K=4$ are shown in Figure~\ref{8}. 
\begin{figure}[htbp]
  \includegraphics[width=0.5\textwidth]{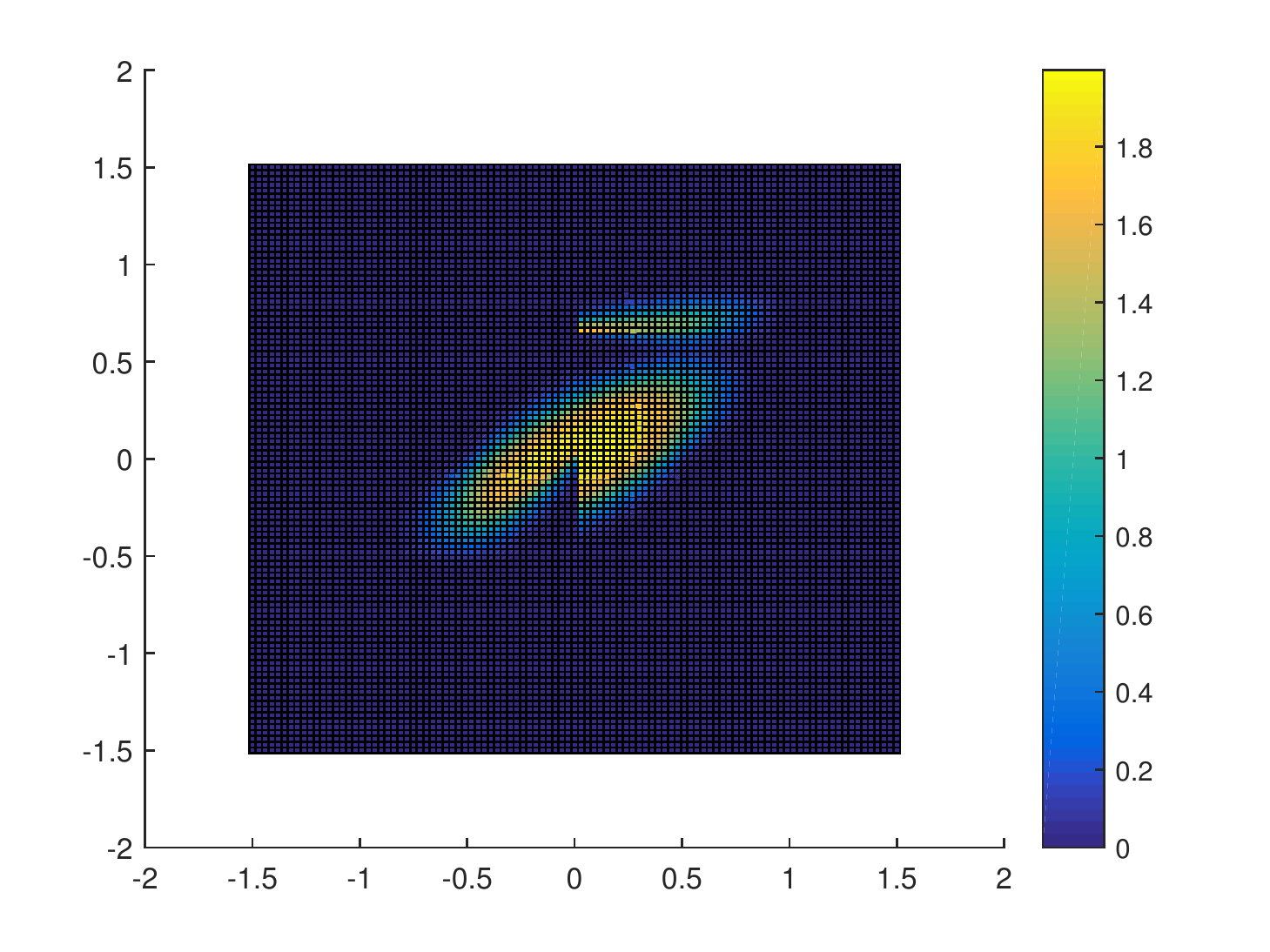}
  \includegraphics[width=0.5\textwidth]{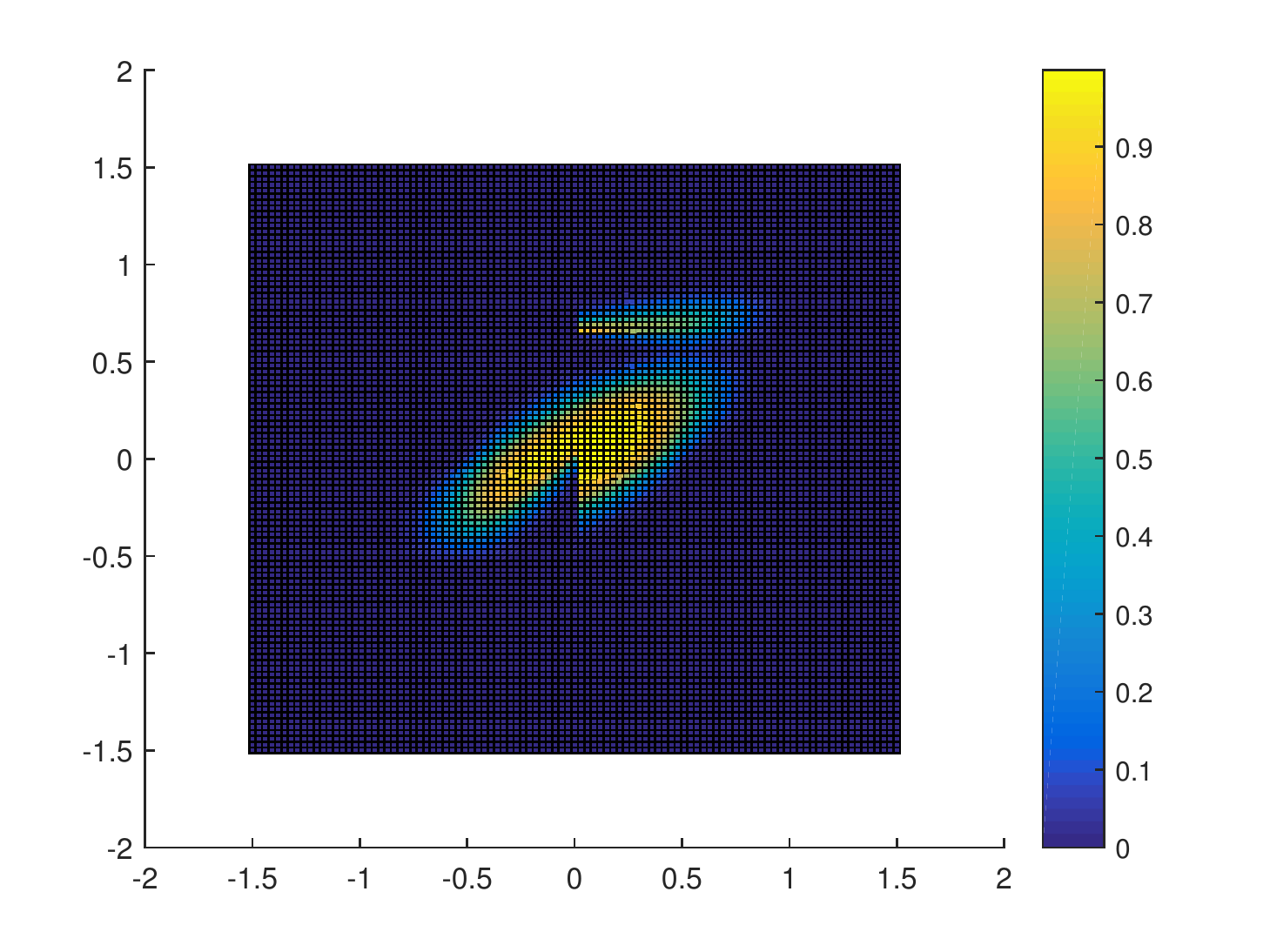}
  \caption{Example 2 with initial data (\ref{ex2-init2}). Expectation of the solution. Left: the collocation method with $20$ sample points. Right: the new gPC-SG method with gPC order $K=4$.}
  \label{8}
\end{figure}
Although the initial data is continuous, due to the interface condition, the solution may still be discontinuous. This singularity will have a big impact on the convergence of gPC method.

\subsubsection{The first order finite difference approximation}
In this example we set the initial data as
\begin{equation}
  u(x,v,0)=
  \begin{cases}
    1, \quad &x\geq 0, v<0, x^2+v^2<1, \\
    1, \quad &x\leq 0, v>0, x^2+v^2<1, \\
    0, \quad &\text{otherwise}.
  \end{cases}
\label{ex2-init1}
\end{equation}
Notice that the solution has singularity due to both the initial data and the discontinuous potential. The deterministic version of this example was used in~\cite{Wen:2005ueba} and the analytic solution can be obtained by using the method of characteristics. We first plot the analytic solution and numerical solution (using the first order flux) with a fixed $z=0$ in Figure~\ref{9} corresponding to the deterministic example in~\cite{Wen:2005ueba}.
\begin{figure}
  \includegraphics[width=0.5\textwidth]{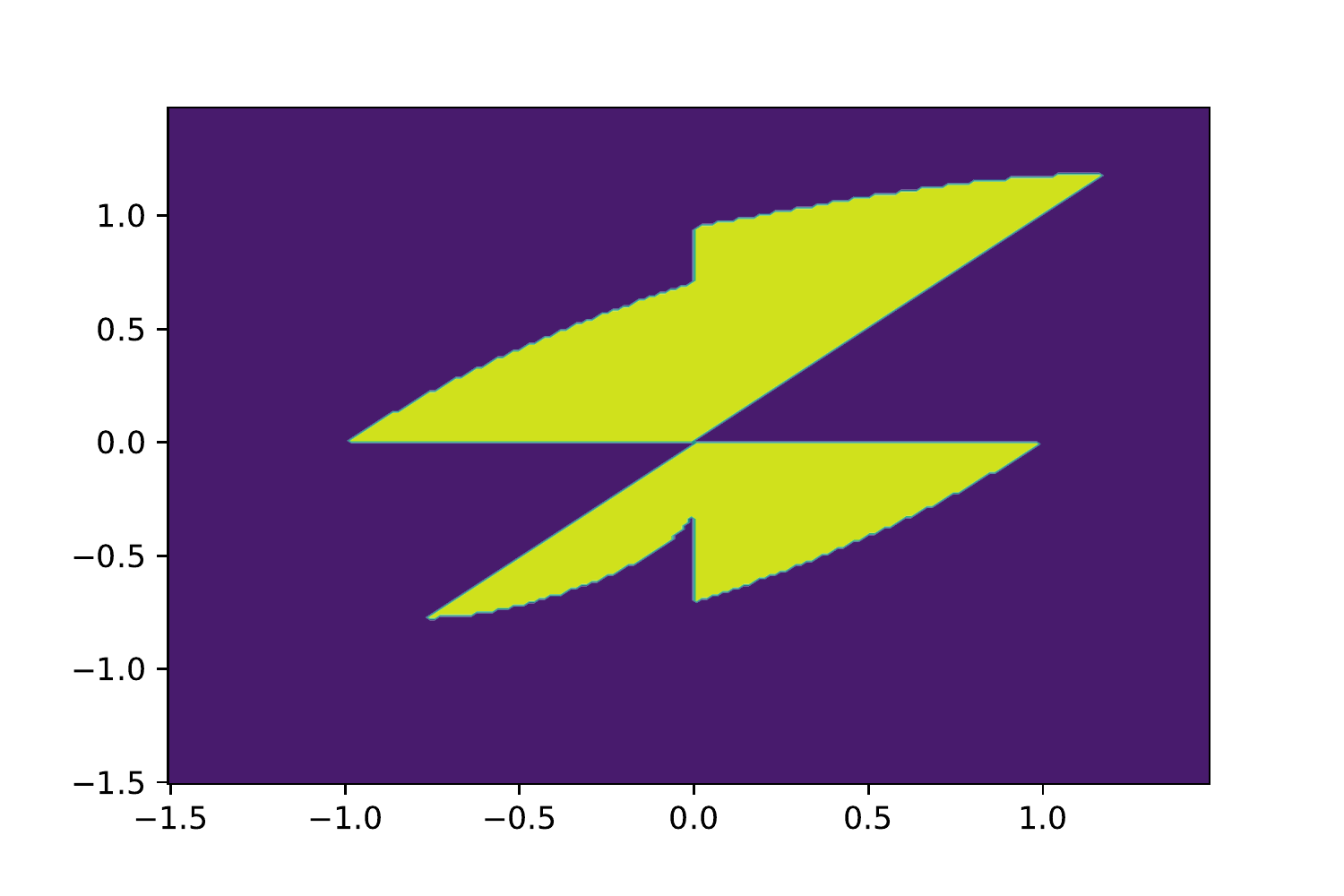}
  \includegraphics[width=0.5\textwidth]{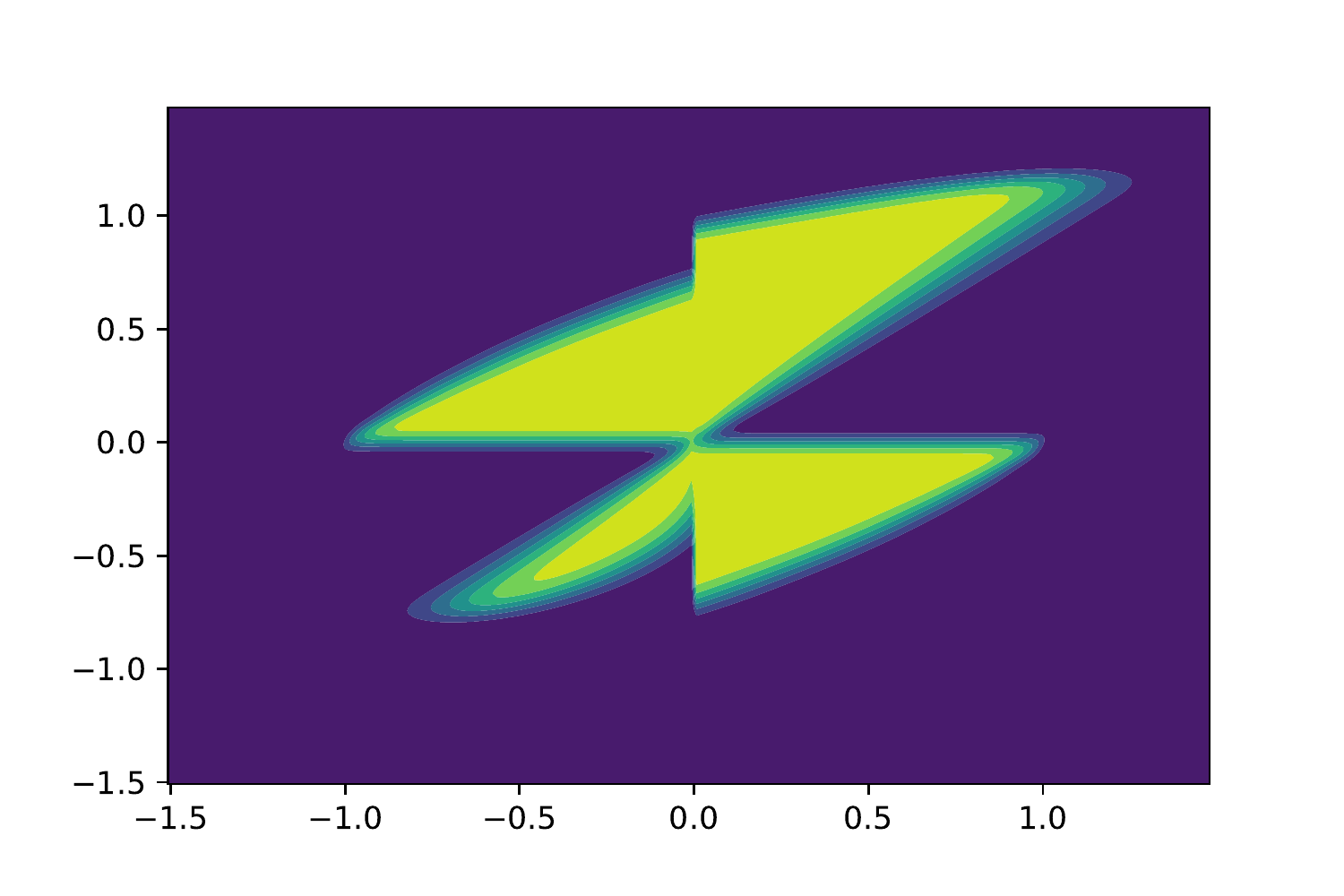}
  \caption{The deterministic case of Example 2 with initial data (\ref{ex2-init1}). Left: analytic solution of the deterministic problem with $z=0$ and $t=1$. Right: numerical solution using the first order Hamiltonian preserving scheme with $\Delta x = \Delta v = 0.015$, $\Delta t = 0.001$}
  \label{9}
\end{figure}

Then we compare the solution computed by the collocation method with $M=20$ sample points (Figures~\ref{6} and~\ref{15} left). Figures~\ref{6} and~\ref{15} right show the solutions  by our new gPC-SG method with gPC order $K=10$. Here the mesh size is $\Dx=\Dv=0.03$ and time step is $\Dt=0.002$. One can see the difference between the expectation of the stochastic solution and the deterministic case when $z=0$ and this differences can be easily seen on the variance plots as well. The expectation of the stochastic solution is expected to be smoother
since it integrates over the $z$ variable, thus gains on order of regularity
(see examples in \cite{Des, HJX}).
For the computation cost, our new gPC-SG method runs much faster than collocation method. The collocation method takes about $20$ times cost of the deterministic version due to $20$ sample points we choose, however, our new gPC-SG,
with $K=10$,  takes about $10$ times the cost of the deterministic problem.

\begin{figure}[htbp]
  \includegraphics[width=0.5\textwidth]{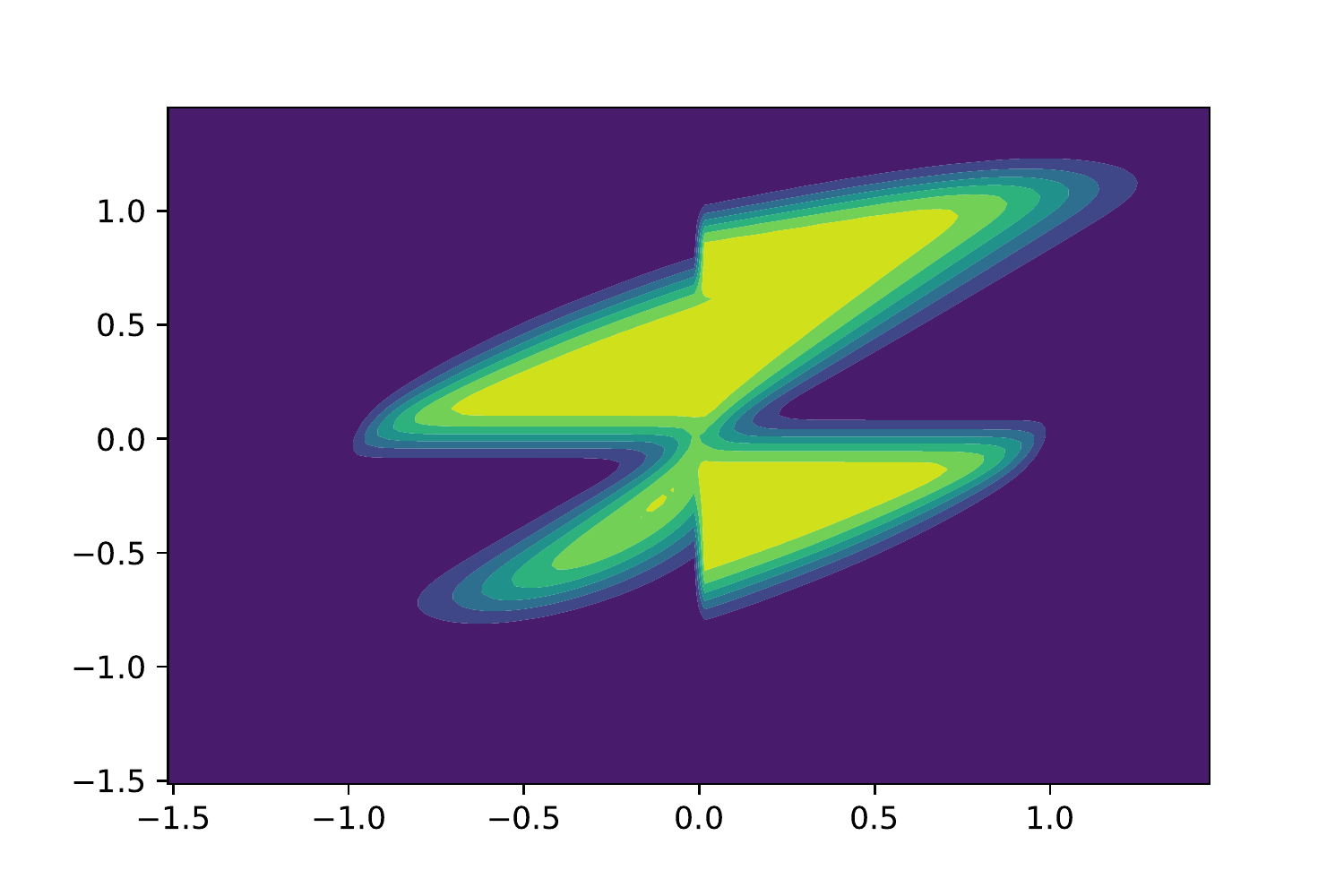}
  \includegraphics[width=0.5\textwidth]{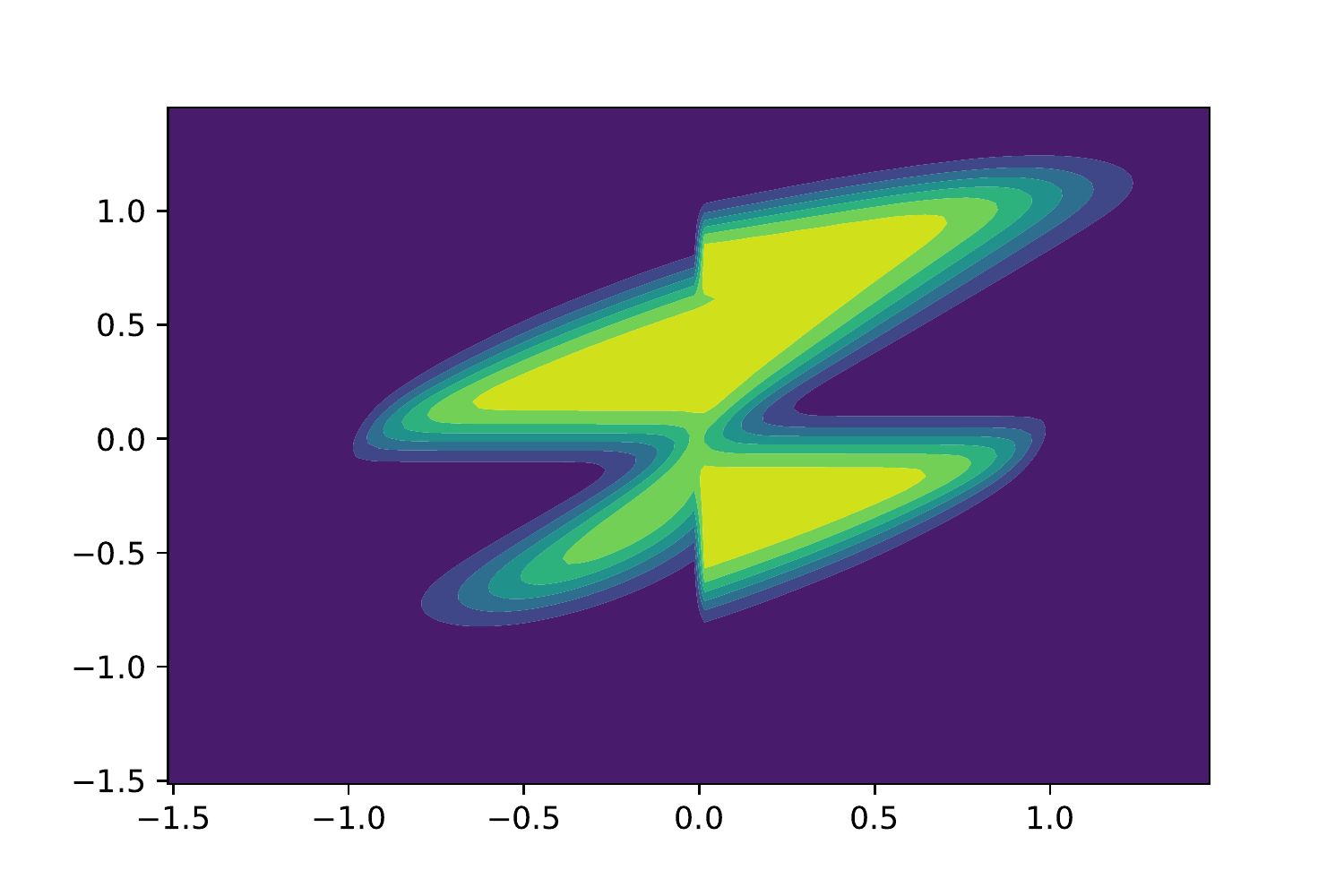}
  \caption{Example 2 with initial data (\ref{ex2-init1})  by the first order 
finite difference approximation  with $\Dx=\Dv=0.03$ and $\Dt=0.002$.  The expectation of the solution. Left: the collocation method with $M=20$ samples points. Right: the new gPC-SG method using first order finite difference approximation.}
  \label{6}
\end{figure}
\begin{figure}[htbp]
  \includegraphics[width=0.5\textwidth]{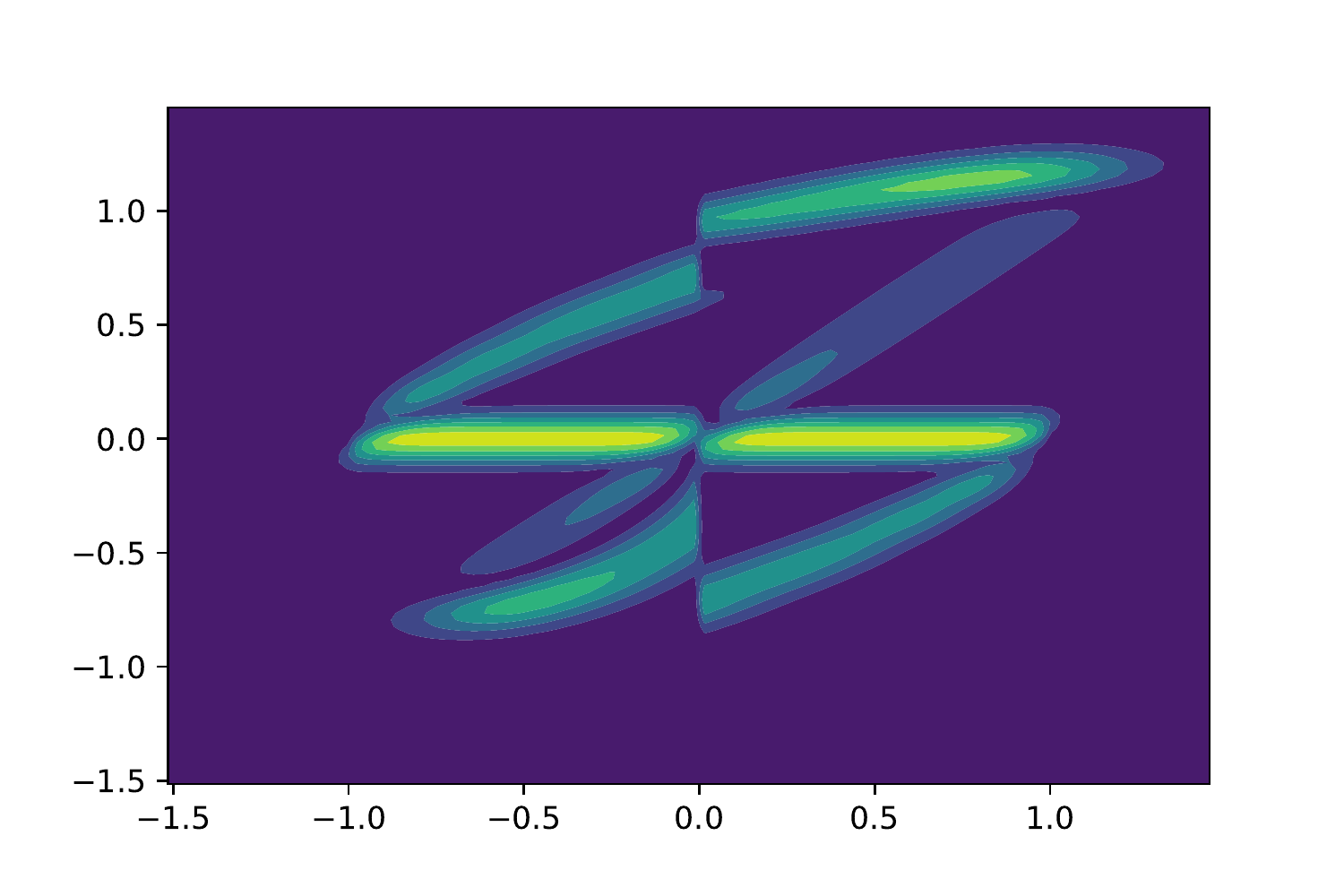}
  \includegraphics[width=0.5\textwidth]{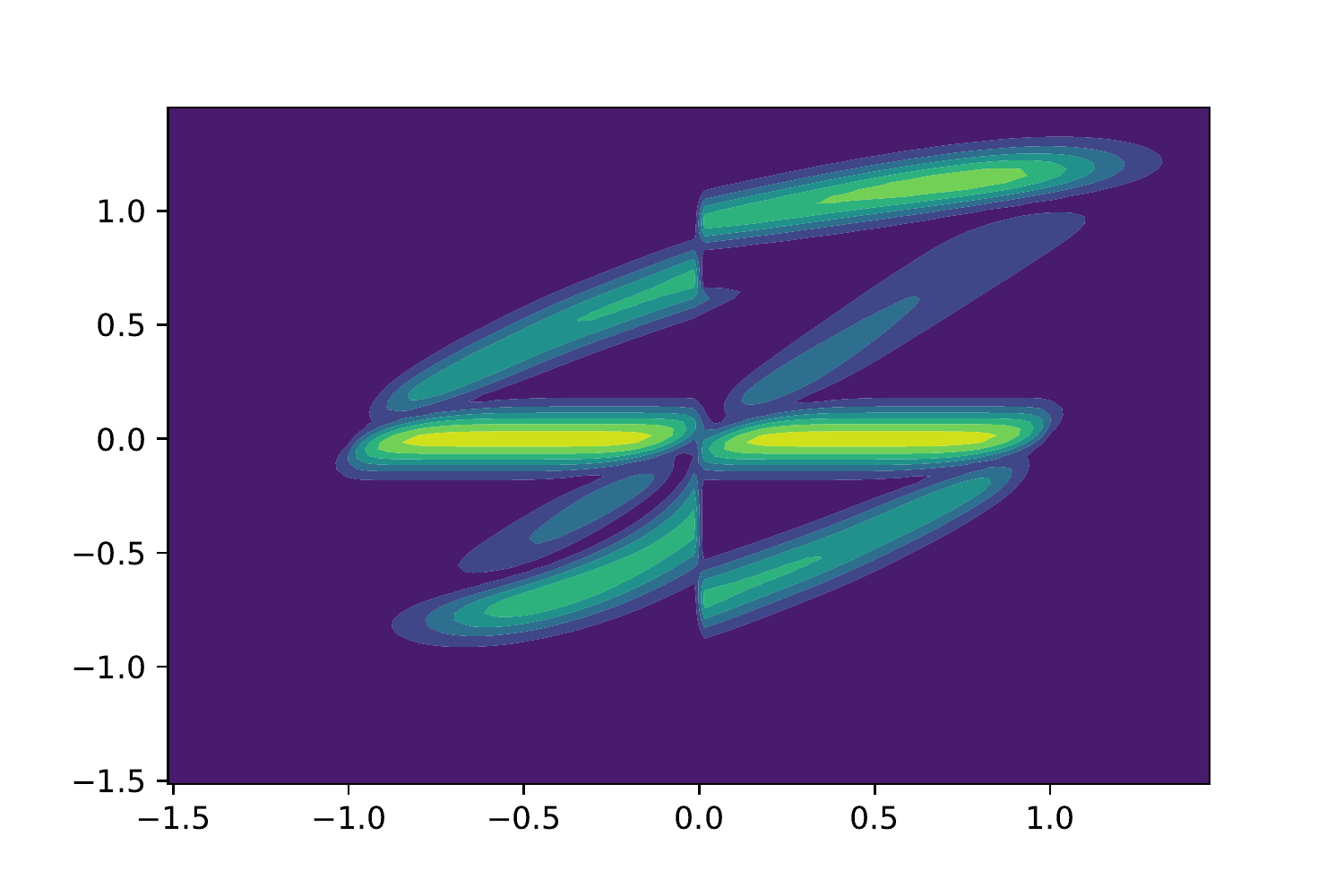}
  \caption{Example 2 with initial data (\ref{ex2-init1}) by the first order 
finite difference approximation  with $\Dx=\Dv=0.03$ and $\Dt=0.002$. The variance of the solution. Left: the collocation method. Right: the new gPC-SG 
method.}
  \label{15}
\end{figure} 

In Figure~\ref{7} we plot the $\ell^1$ error of the discrete gPC-SG method as the gPC order $K$ increases. This figure shows the spectral convergence.
\begin{figure}[htbp]
  \includegraphics[width=0.5\textwidth]{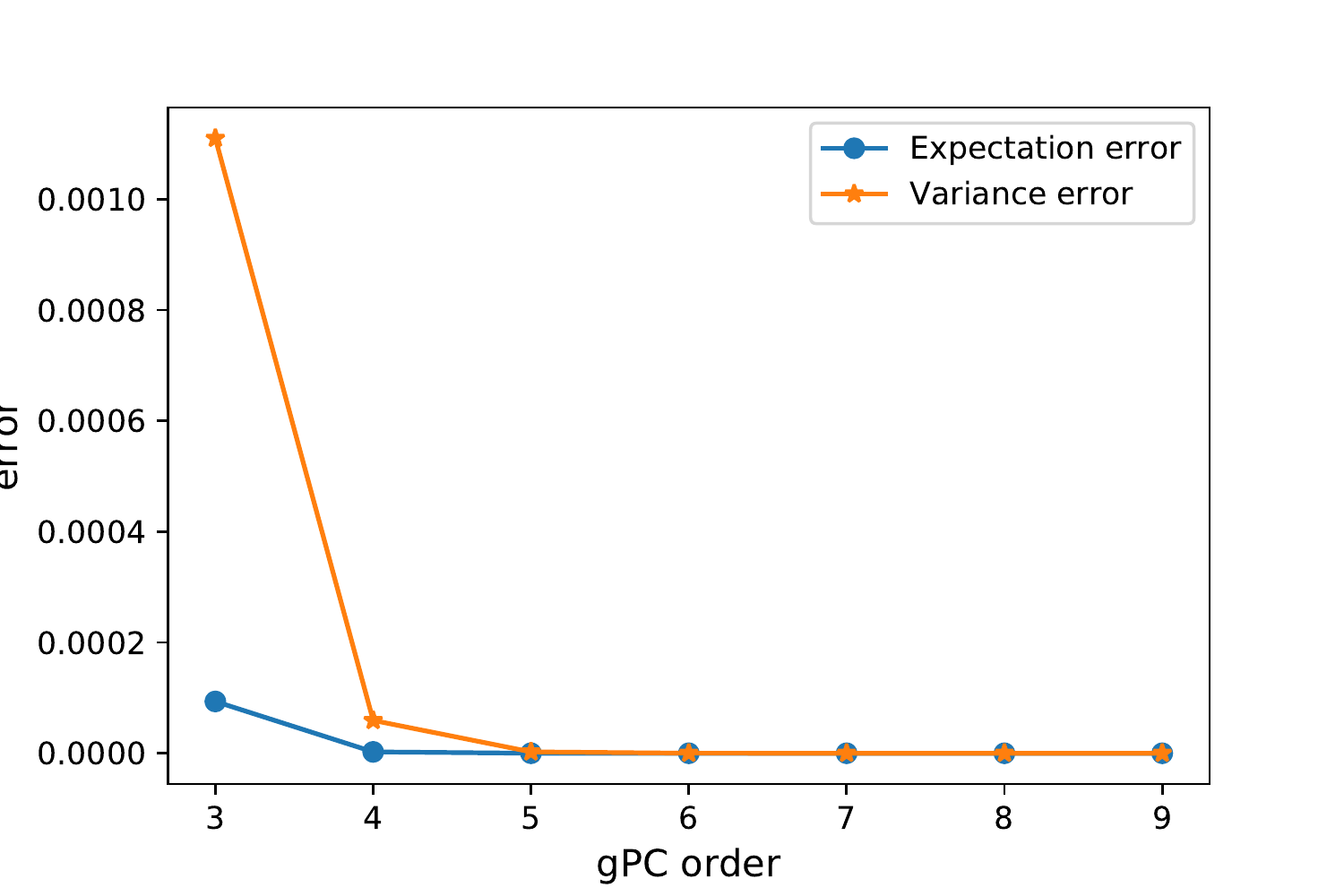}
  \includegraphics[width=0.5\textwidth]{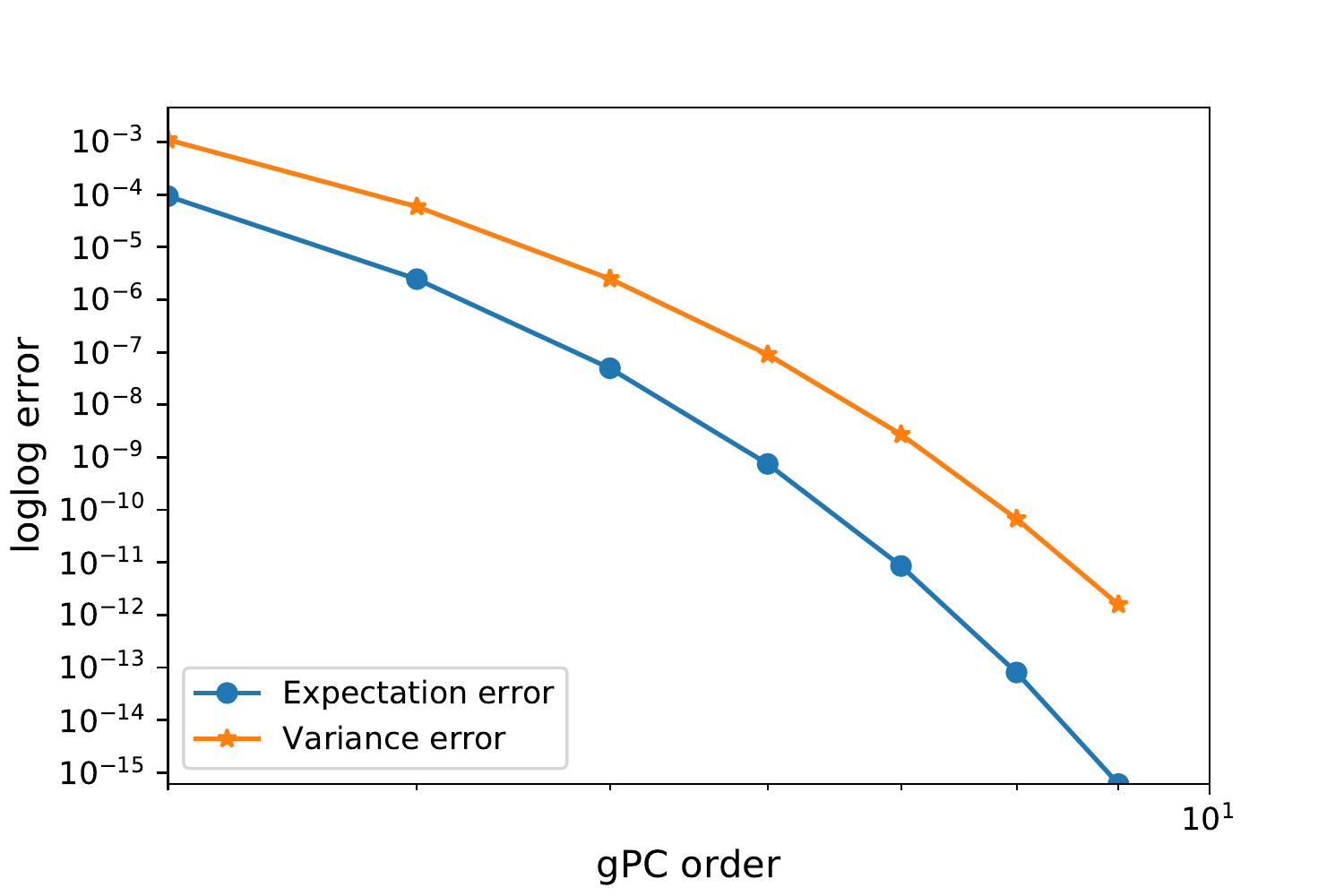}
  \caption{Example 2 with initial data (\ref{ex2-init1}). Convergence of the new gPC-SG method using the first order finite difference approximation. Left: the $\ell^1$ error versus gPC order. Right: the gPC error versus the gPC order by a log-log plot (with other numerical parameters fixed).}
  \label{7}
\end{figure}

\subsubsection{The second finite difference approximation}
Here for the second finite difference approximation we still use the same set up as in previous subsection for the first order case. 

First as in the first order case, we will show the numerical solution of the deterministic case when $z=0$ using the second order flux in Figure~\ref{deter_sec}.  The second order method clearly gives a much sharper resolution for the discontinuities than the first order method (compare the right figures of
Figure 10 and Figure 15). But due to the Lax-Wendroff scheme we use in the $v$-direction~(\ref{vflux}), there exists some oscillations due to numerical dispersion.
\begin{figure}
  \includegraphics[width=0.5\textwidth]{exact}
  \includegraphics[width=0.5\textwidth]{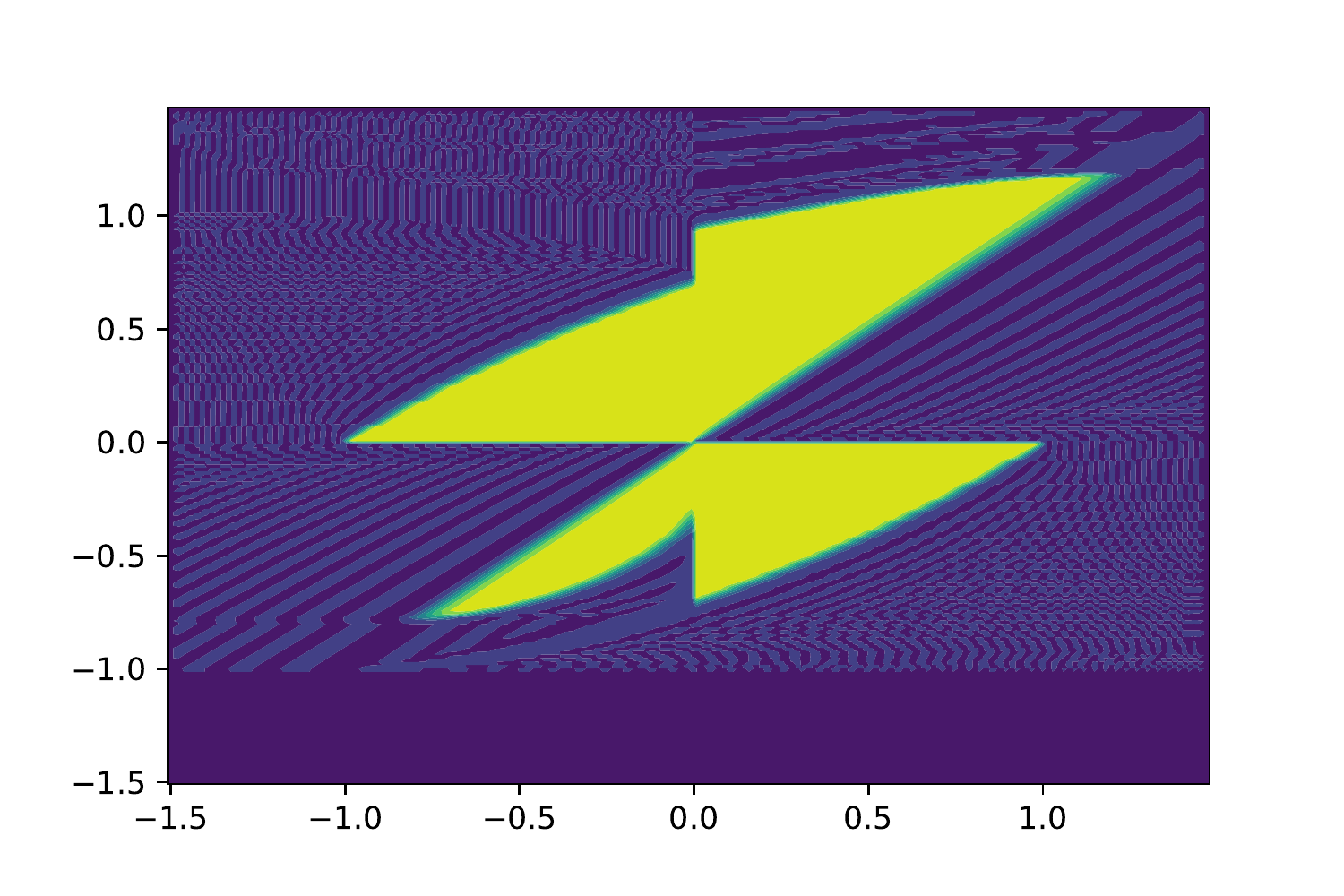}
  \caption{The deterministic case of Example 2 with initial data~(\ref{ex2-init1}). Left: analytic solution of the deterministic problem with $z=0$ and $t=1$. Right: numerical solution using the second order Hamiltonian preserving scheme with $\Delta x = \Delta v = 0.015$, $\Delta t = 0.001$}
  \label{deter_sec}
\end{figure}

Then we will show the expectation and variance of the solution calculated by the collocation method with $M=20$ sample points and our new gPC-SG method (for the calculation of $\left<\mathrm{RHS}(z)\right>$ we also use 20 Gauss-Legendre quadrature points). See Figure~\ref{11} and Figure~\ref{12}.
\begin{figure}
  \includegraphics[width=0.5\textwidth]{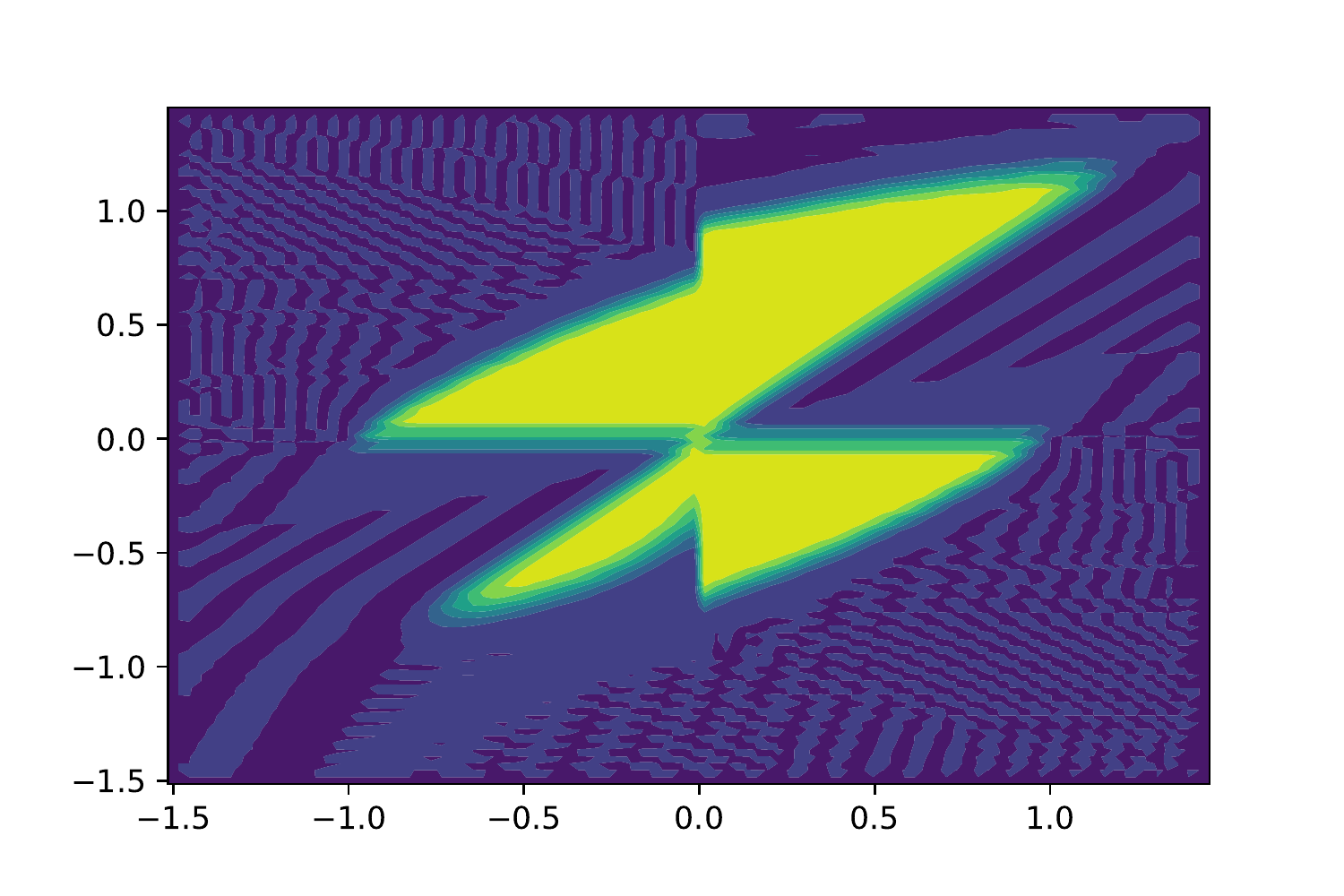}
  \includegraphics[width=0.5\textwidth]{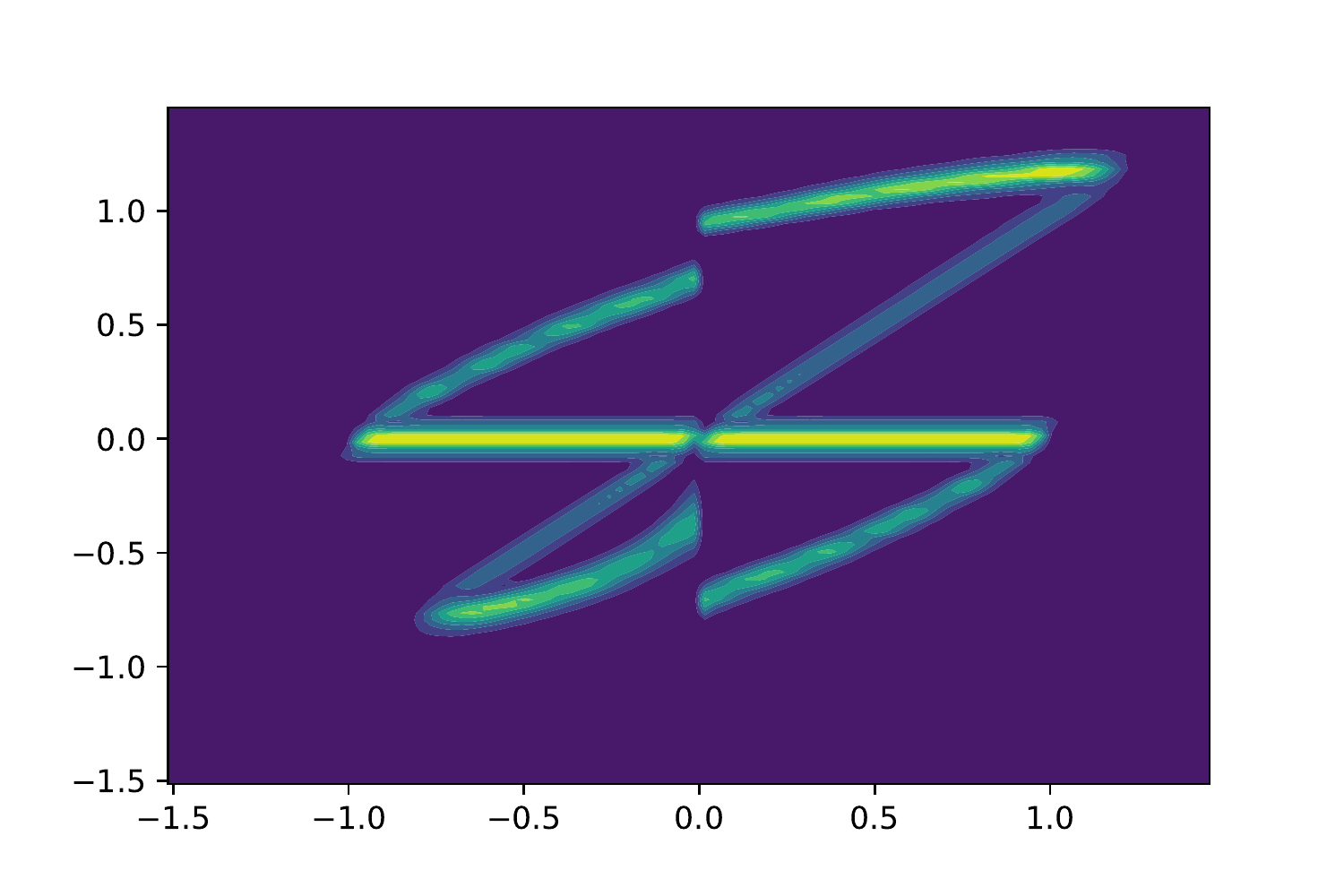}
  \caption{Example 2 with initial data (\ref{ex2-init1}) by the second order 
finite difference approximation  with $\Delta x = \Delta v = 0.03$, $\Delta t = 0.002$. Reference solution 
 by the collocation method with 20 sample points at $t=1$. Left: expectation. Right: variance.}
  \label{11}
\end{figure}
\begin{figure}
  \includegraphics[width=0.5\textwidth]{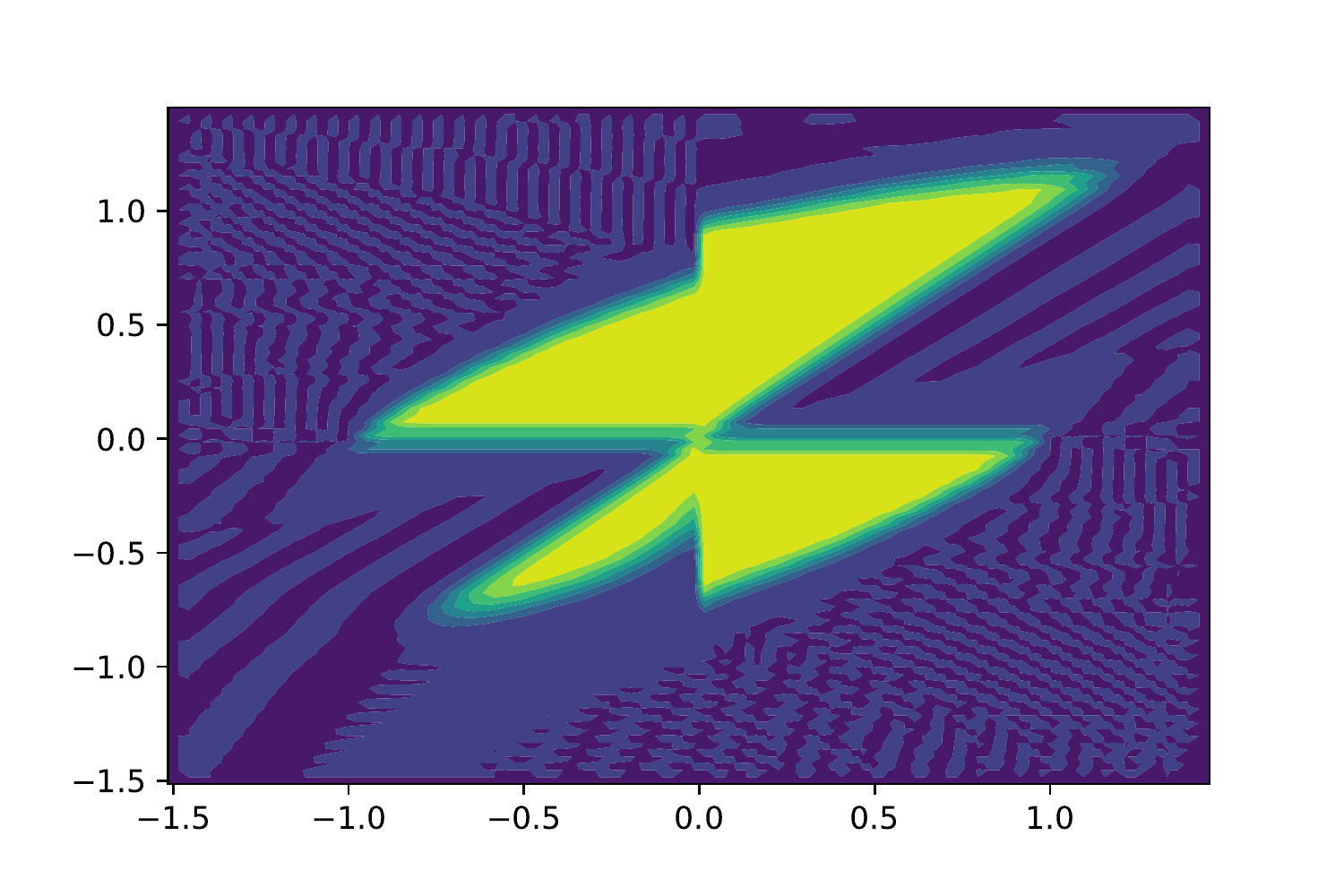}
  \includegraphics[width=0.5\textwidth]{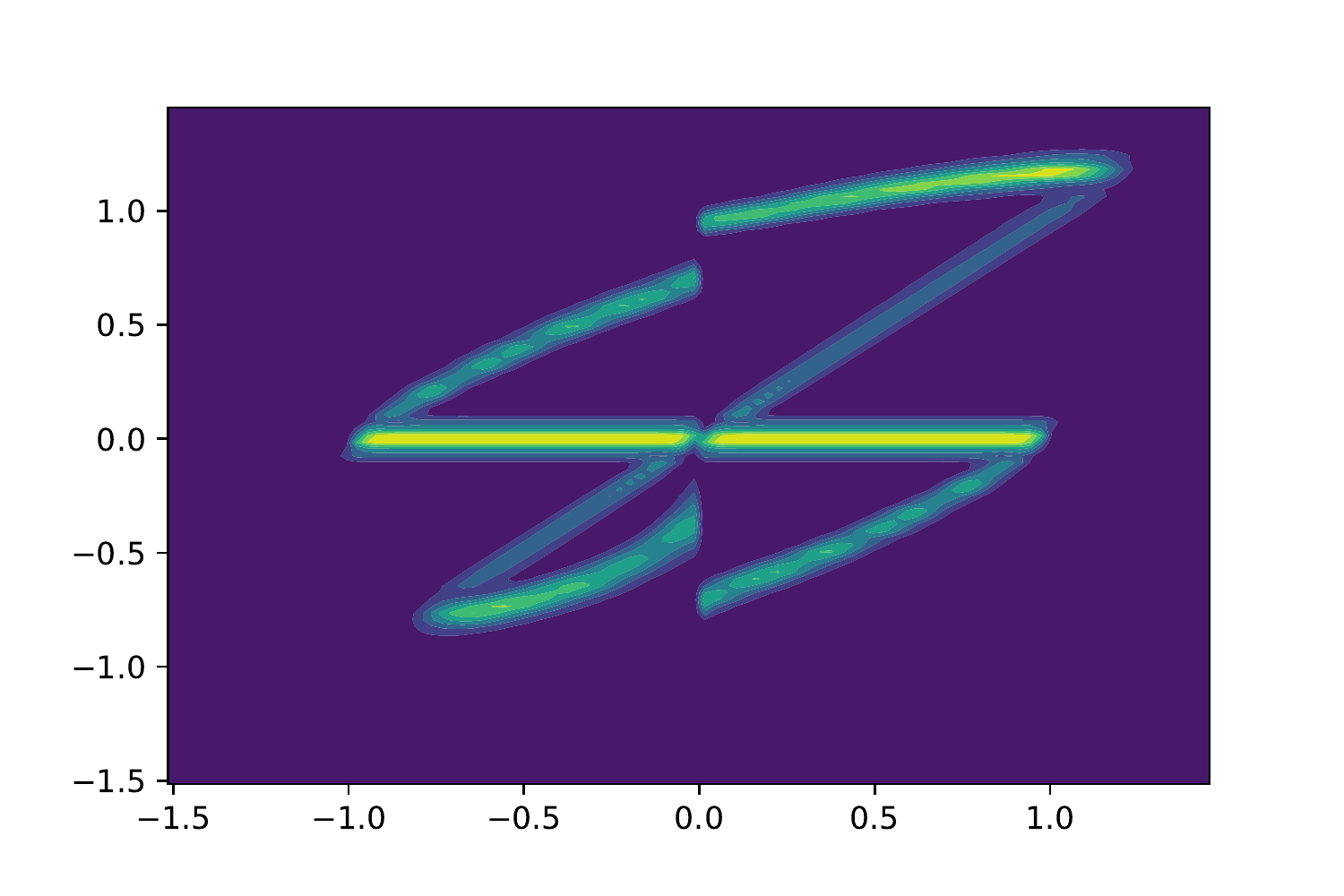}
  \caption{Example 2 with initial data (\ref{ex2-init1}) by the second order 
finite difference approximation  with $\Delta x = \Delta v = 0.03$, $\Delta t = 0.002$. Solution at $t=1$ computed by the new gPC-SG method. Left: expectation. Right: variance.}
  \label{12}
\end{figure}
One can find no difference between these two methods, both giving sharper
resolutions at discontinuities than their first order counterparts shown in
Figures 11 and 12. Here for the computation cost, we point out that unlike in the first order case, our new gPC-SG method runs only slightly faster than the collocation method since in the calculation of $\left<\mathrm{RHS}(z)\right>$ we use a similarly technique as the collocation method.

Finally, we test the convergence rate of our new gPC-SG method. To do this, we first fix our mesh size: $\Delta x = \Delta v = 0.03$, $\Delta t = 0.02$, and output the result at $t=1$. We  use $20$ Gauss-Legendre quadrature points to compute the inner product in (\ref{gPC2}). We choose the gPC order $K=10$ as our reference solution, and see how the error changes when increasing $K$ from $3$ to $10$. From Figure~\ref{14}, an exponential convergence can be observed.
\begin{figure}
  \includegraphics[width=0.5\textwidth]{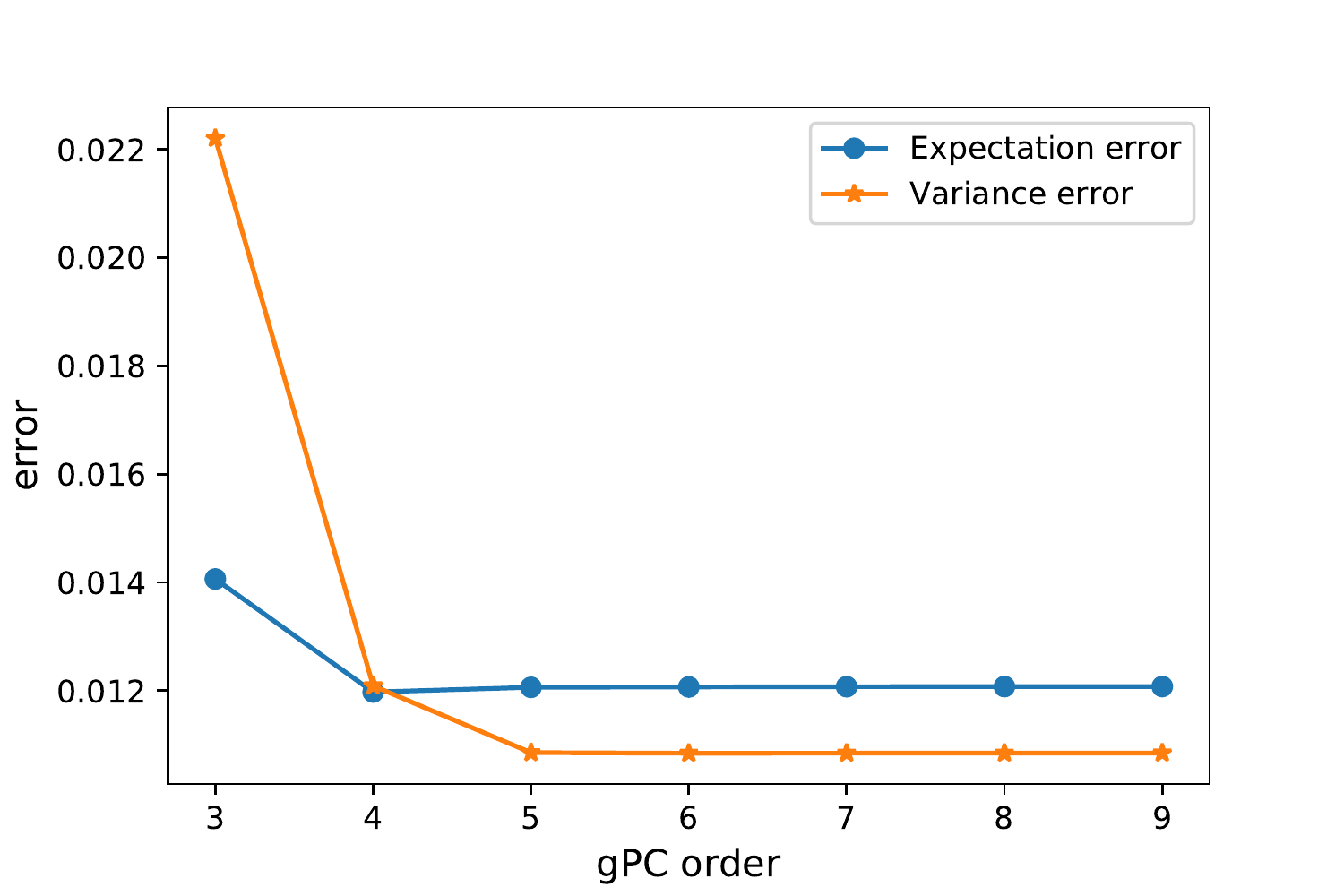}
  \includegraphics[width=0.5\textwidth]{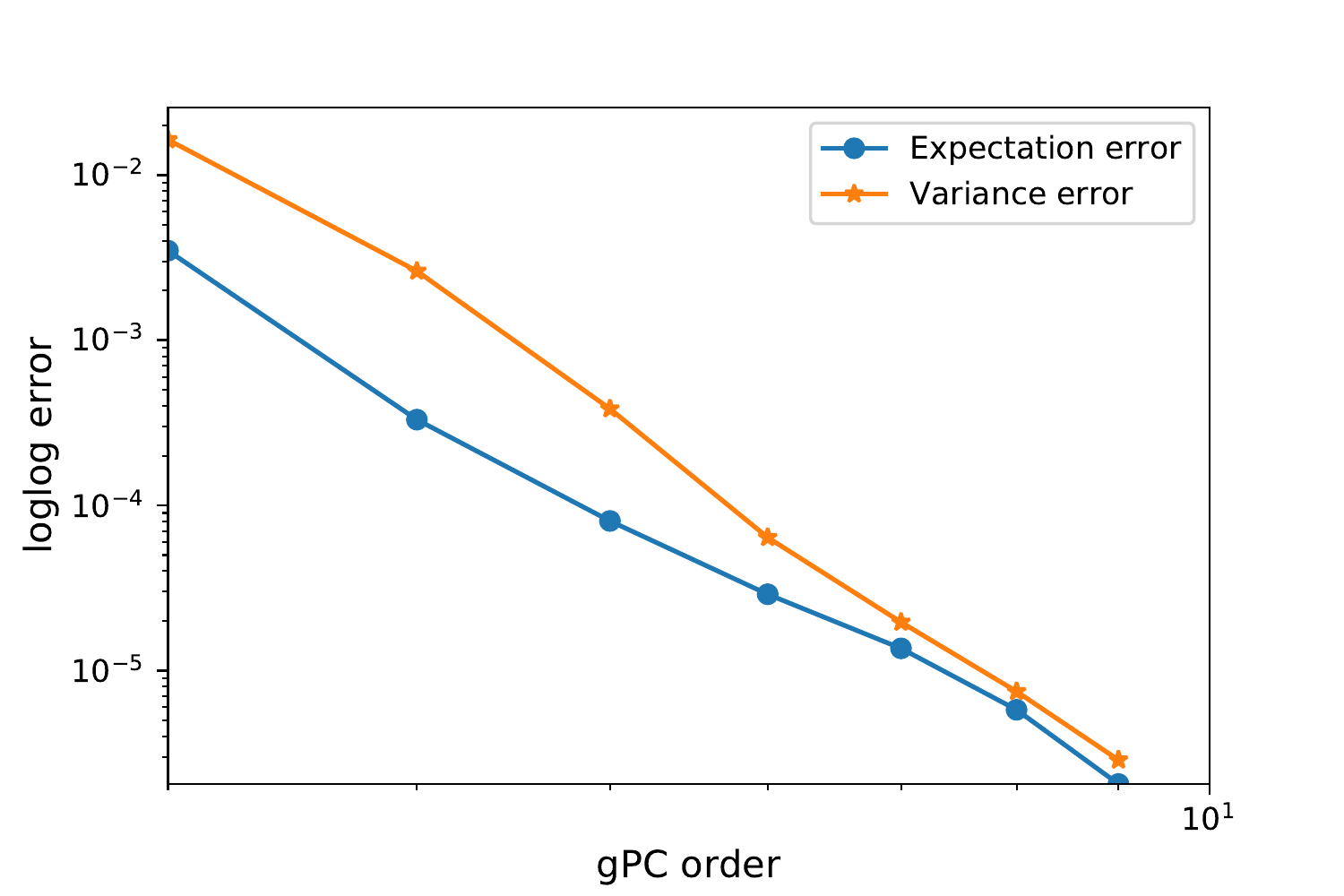}
  \caption{Example 2 with initial data (\ref{ex2-init1}). Convergence of the new gPC-SG method using second order finite difference approximation. Left: the $\ell^1$ error versus the gPC order. Right: the gPC error versus the gPC order by a log-log plot (with other numerical parameters fixed).}
  \label{14}
\end{figure}

\end{document}